\documentclass[12pt,a4paper,reqno]{amsart}
\usepackage{dutchcal}
\usepackage[T1]{fontenc}
\usepackage{phaistos}
\usepackage{protosem}
\usepackage{amsmath}
\allowdisplaybreaks
\usepackage{amsthm}
\usepackage{amsfonts}
\usepackage{amssymb}
\usepackage{graphicx}
\usepackage{color}
\usepackage{amsbsy}
\usepackage{mathrsfs}
\usepackage{bbm}
\usepackage{verbatim}
\usepackage{tikz}
\usepackage{booktabs}
\usepackage{pgfplots} % 用于计算和坐标，简化 89 度斜坡绘制
\pgfplotsset{compat=1.18}
\usetikzlibrary{shapes,arrows,positioning}

\usepackage[colorlinks=true,linkcolor=black,citecolor=blue]{hyperref}
\newcommand{\R}{\mathbb R}

\newtheorem{theorem}{Theorem}[section]
\newtheorem{lemma}[theorem]{Lemma}

\newtheorem{corollary}[theorem]{Corollary}

\theoremstyle{remark}
\newtheorem{remark}[theorem]{\it \bf{Remark}\/}

\numberwithin{equation}{section}
\catcode`@=11
\def\section{\@startsection{section}{1}%
  \z@{1.5\linespacing\@plus\linespacing}{.5\linespacing}%
  {\normalfont\bfseries\large\centering}}
\catcode`@=12

\newcommand{\be}{\begin{equation}}
\newcommand{\ee}{\end{equation}}
\newcommand{\bea}{\begin{eqnarray}}
\newcommand{\eea}{\end{eqnarray}}
\newcommand{\bee}{\begin{eqnarray*}}
\newcommand{\eee}{\end{eqnarray*}}

\def\RR{\mathbb{R}}

\catcode`@=11
\def\supess{\mathop{\operator@font Sup\,ess}}
\catcode`@=12

\def\RR{\mathbb{R}}

\def\R2+{\RR ^2_+}

\def\lim{\mathop{\rm lim}}

\def\sup{\mathop{\rm sup}}
\def\exp{{\rm exp}}

\def\log{{\rm log}}

\begin{document}

\title[]{FORMATION OF QUASI-SINGULARITIES FOR COMPRESSIBLE EULER FLOWS}
\author[]{XIELING FAN}
\address{School of Mathematics, Central South University, Changsha, China and Department of Mathematics, City University of Hong Kong, Kowloon, Hong Kong SAR, China}
\email{fanxieling@outlook.com, xielinfan2-c@my.cityu.edu.hk}
\author[]{Hongyu Liu}
\address{Department of Mathematics, City University of Hong Kong, Kowloon, Hong Kong SAR, China}
\email{hongyu.liuip@gmail.com, hongyuliu@cityu.edu.hk}
\begin{abstract}
This paper establishes a novel quasi-singularity formation phenomenon for the three-dimensional isentropic compressible Euler equations. For any prescribed finite collections of pairwise distinct points $\{z^{(1)}_j\}_{j=1}^N, \{z^{(2)}_j\}_{j=1}^N \subset \mathbb{R}^3$, an arbitrarily large parameter $M \gg 1$, and an arbitrary time $T>0$, we explicitly construct real-analytic initial data and an external force generated by Herglotz-type wave fields together with a suitably chosen damping parameter. We prove that the corresponding unique classical solution exists on $[0,T]$ and exhibits the following properties: first, the pressure remains uniformly bounded from below and above proportionally to the constant initial pressure; second, every component of the velocity and its first-order spatial derivatives simultaneously undergo high-amplitude (exceeding $M$) and high-frequency oscillations at all points $z^{(1)}_j$ ($j=1,\dots,N$) throughout $[0,T]$; and third, the velocity and its first-order spatial derivatives remain uniformly bounded by $1$ at all points $z^{(2)}_j$ ($j=1,\dots,N$) throughout $[0,T]$. 

The construction exhibits two complementary features: localized
large-amplitude oscillations of the velocity and its spatial
gradient at prescribed points, and a pronounced contrast between
the two groups of points. Since the solution remains smooth, this
behavior may be viewed as an intermediate regime between moderate
regular flows and genuine finite-time blowup. Moreover, the
initial data, external force, and damping parameter are specified
explicitly, providing a concrete framework for investigating such
behavior in controlled compressible flows. The proof combines the
construction of Herglotz wave fields with a suitable choice of
the damping parameter and quantitative energy estimates for the
nonlinear remainder.


\end{abstract}

\maketitle

%%%%%%%%%%%%%%%%%%%%%%%%%%%%%%%%%%%%%%%%%%%%%%%%%%
%%%%%%%%%%%%%%%%%%%%%%%%%%%%%%%%%%%%%%%%%%%%%%%%%%

\section{Introduction}

%%%%%%%%%%%%%%%%%%%%%%%%%%%%%%%%%%%%%%%%%%%%%%%%%%
%%%%%%%%%%%%%%%%%%%%%%%%%%%%%%%%%%%%%%%%%%%%%%%%%%

%%%%%%%%%%%%%%%%%%%%%%%%%%%%%%%%%%%%%%%%%%%%%%%%%%

\subsection{Setting of the problem}

%%%%%%%%%%%%%%%%%%%%%%%%%%%%%%%%%%%%%%%%%%%%%%%%%%
We consider the three-dimensional isentropic compressible Euler equations formulated in terms of the pressure $p$ and velocity $v=(v^{(1)},v^{(2)},v^{(3)})$:
\begin{equation}\label{eq:euler_pressure_velocity}
\begin{cases}
\partial_t p + v \cdot \nabla p + \gamma p \, \nabla \cdot v = 0, & \text{in } (0,T) \times \mathbb{R}^{3}, \\[1.2ex]
\partial_t v + v \cdot \nabla v + \dfrac{1}{\rho(p)} \nabla p + \kappa v = f, & \text{in } (0,T) \times \mathbb{R}^{3}, \\[2ex]
p(x,0) = p_{\mathrm{in}}(x), \quad v(x,0) = v_{\mathrm{in}}(x), & \text{in } \mathbb{R}^{3},
\end{cases}
\end{equation}
where $T>0$ is a prescribed time and $\gamma > 1$ denotes the adiabatic exponent. The fluid density $\rho$ and pressure $p$ satisfy the equation of state
\begin{equation}\label{eq:pressure-density-relation}
p = A \rho^{\gamma}, \quad \text{or equivalently,} \quad \rho(p) = \left(\frac{1}{A}\right)^{\frac{1}{\gamma}} p^{\frac{1}{\gamma}},
\end{equation}
where $A > 0$ is a constant, $\kappa > 0$ is a damping parameter, and $f$ denotes a smooth external force. The smooth initial data $(p_{\mathrm{in}}, v_{\mathrm{in}})$, together with the damping parameter $\kappa$ and the external force $f$, will be carefully constructed in the section \ref{sec:acoustic}.

The problem of understanding singularity formation in fluid dynamics is both fundamental and challenging. A central aspect is whether smooth initial data can evolve into a regime where regularity is lost and certain physical or geometric quantities blow up in finite time, and how such singular behavior develops and localizes. For the compressible Euler equations, extensive research has been devoted to various singularity formation mechanisms, including shock formation and, more recently, implosion type singularities. These results concern genuine singularities, where quantities become unbounded and regularity breaks down.

In this paper, we investigate a different but closely related
regime for the three-dimensional isentropic compressible Euler
equations. For any prescribed time interval $[0,T]$, we construct
classical solutions whose velocity and spatial gradient exhibit
localized, high-amplitude and high-frequency oscillations of
arbitrarily large magnitude at prescribed points, while the
solutions remain smooth. Moreover, these quantities exhibit a
pronounced spatial contrast between two prescribed collections of
points. This behavior provides a perspective on the transition
between regular flows and genuine finite-time blowup. Since the
initial data, damping parameter, and external force are specified
explicitly, the construction also provides a concrete framework
for controlling the behavior of compressible flows. We refer to
this phenomenon as \emph{quasi-singularity} formation. A precise
description is given in Section~\ref{sec 1.3}.
%%%%%%%%%%%%%%%%%%%%%%%%%%%%%%%%%%%%%%%%%%%%%%%%%%
\subsection{Classical singularity formation for compressible Euler equations}

The compressible Euler equations constitute one of the fundamental
models in fluid dynamics, and the formation of singularities has
been a central topic of research for several decades. Broadly
speaking, singularities in compressible Euler flows include shock
formation, in which the primary physical variables, such as
density, velocity, and pressure, remain bounded while their spatial
derivatives become unbounded. Another type of singular behavior
involves the blowup of the physical variables themselves. Such
singularities may arise through concentration mechanisms, including
implosion, in which the flow develops increasingly localized
structures.

The mathematical study of singularity formation for compressible
Euler flows dates back to the classical works of Riemann
\cite{Riemann1860}, Lax \cite{Lax1964,Lax1972}, and John
\cite{John1974}, which developed the analysis of nonlinear wave
steepening and shock formation in one-dimensional hyperbolic
systems. In multiple dimensions, finite-time singularity formation
has also been extensively investigated for incompressible fluid
models, providing important context and analytical contrast. Early
numerical investigations by Luo and Hou
\cite{LuoHou2014SIAM,LuoHou2014PNAS} suggested potentially singular
behavior in three-dimensional axisymmetric incompressible Euler
flows with boundaries. More recently, a series of preprints—several
developed with the assistance of large language models—have
proposed finite-time blowup constructions for three-dimensional
incompressible flows, both with and without external forcing:
OpenAI~\cite{OpenAI2026ForcedNS} proposed a smooth, compactly
supported external force driving zero initial velocity to
finite-time blowup for the Navier--Stokes equations; Alp\"oge and
Buckmaster~\cite{AlpogeBuckmaster2026ForcedEuler} proposed a
singularity construction for the forced incompressible Euler
equations via multiscale oscillating interactions; and subsequent
announcements by Alp\"oge and
Buckmaster~\cite{AlpogeBuckmaster2026UnforcedNS} and by Alp\"oge,
Buckmaster, and
Coiculescu~\cite{AlpogeBuckmasterCoiculescu2026UnforcedEuler}
addressed unforced three-dimensional Navier--Stokes and
incompressible Euler blowup, respectively. While these recent
developments concern incompressible fluids, the compressible
setting exhibits distinct physical and analytical mechanisms,
which we review below.

For multidimensional compressible Euler flows, Sideris
\cite{Sideris1985} proved finite-time breakdown of smooth solutions
for suitable initial data. The geometric structure of
multidimensional shock formation was later clarified by
Christodo-ulou \cite{Christodoulou2007}, who showed that suitable
smooth initial data can develop a shock surface in finite time:
the solution remains bounded up to the shock, while its
first-order spatial derivatives become singular. More recently,
Buckmaster, Shkoller, and Vicol
\cite{BuckmasterShkollerVicol2022} constructed stable point shock
formation for the three-dimensional isentropic compressible Euler
equations. Their construction gives precise information about the
blowup time and location and describes the H\"older regularity of
the solution at the singular point.

Beyond shock formation, Chen, Cialdea, Shkoller, and Vicol
\cite{ChenCialdeaShkollerVicol2024}, together with Chen
\cite{Chen2024HigherD}, studied finite-time vorticity blowup for
the compressible Euler equations through implosion mechanisms.
Merle, Rapha\"el, Rodnianski, and Szeftel
\cite{MerleRaphaelRodnianskiSzeftel2022a,MerleRaphaelRodnianskiSzeftel2022b}
constructed implosion solutions for compressible Euler equations
in which density and velocity become unbounded in finite time, and
analyzed the associated self-similar blowup structure. Developing
this direction further, Buckmaster, Cao-Labora, and
G\'omez-Serrano \cite{buckmaster2025smooth} used a rigorous
computer-assisted approach to study smooth imploding solutions for
three-dimensional compressible fluids under spherical symmetry.
Cao-Labora, G\'omez-Serrano, Shi, and Staffilani
\cite{CaoLabora2025CJM} subsequently studied non-radial implosion
singularities for compressible Euler and Navier--Stokes equations,
extending the analysis beyond the spherically symmetric setting.

The study of fluid singularities also connects with Tao's broader
program on nonlinear dynamics and the limitations of energy
estimates. Tao established finite-time blowup for supercritical
defocusing nonlinear Schr\"odinger equations
\cite{tao2016finite}, and constructed finite-time blowup solutions
for an averaged three-dimensional Navier--Stokes equation
\cite{tao2016finite_ns} and for Lagrangian modifications of the
three-dimensional Euler equation \cite{tao2016finite_euler}.
He also studied connections between dynamical systems and
computation \cite{tao2017universality}, and discussed their
possible relevance to the search for fluid singularities
\cite{tao2019searching}. Complementary to these blowup results,
quantitative regularity estimates and global smooth regimes for
dispersive and hydrodynamic equations have been investigated in,
for example, \cite{tao2008quantitative,ionescu2019wavekg,
deng2017gravity3d}.

%%%%%%%%%%%%%%%%%%%%%%%%%%%%%%%%%%%%%%%%%%%%%%%%%%

\subsection{Quasi-singularity formation phenomenon}
\label{sec 1.3}

%%%%%%%%%%%%%%%%%%%%%%%%%%%%%%%%%%%%%%%%%%%%%%%%%%
In contrast to classical singularity formation scenarios, we
introduce a phenomenon for the isentropic compressible Euler
equations~\eqref{eq:euler_pressure_velocity}, which we refer to
as \emph{quasi-singularity formation}. On any prescribed time
interval $[0,T]$, the velocity and its spatial gradient exhibit
arbitrarily large, localized oscillations at prescribed points,
while the solution remains smooth. The construction also provides
a framework for controlling this behavior through explicitly
specified initial data, an external force, and a damping
parameter.

More precisely, let $N\geq1$ be an integer, and let
\[
 \{z_j^{(1)}\}_{j=1}^N\subset\mathbb R^3,
 \qquad
 \{z_j^{(2)}\}_{j=1}^N\subset\mathbb R^3
\]
be two prescribed collections whose union consists of pairwise
distinct points. For any sufficiently large parameter $M\gg1$
and any prescribed time $T>0$, we construct real-analytic initial
data and an external force using Herglotz wave fields, together
with a suitable damping parameter $\kappa$. The corresponding
unique classical solution exists on $[0,T]$ and has the following
properties:

\begin{enumerate}
\item Each component of the velocity $v$ and all of its first-order spatial derivatives simultaneously undergo localized amplification and high-frequency oscillations at all points $z^{(1)}_j$ ($j=1,\dots,N$):
\begin{align*}
v^{(i)}\!\left(z^{(1)}_j, \frac{2\pi m}{\omega_0}\right) &> M, &
v^{(i)}\!\left(z^{(1)}_j, \frac{2\pi m+\pi}{\omega_0}\right) &< -M, \\
\partial_{x_k}v^{(i)}\!\left(z^{(1)}_j, \frac{2\pi m}{\omega_0}\right) &> M, &
\partial_{x_k}v^{(i)}\!\left(z^{(1)}_j, \frac{2\pi m+\pi}{\omega_0}\right) &< -M,
\end{align*}
for all $i,k \in \{1,2,3\}$, $j \in \{1,\dots,N\}$, and every integer $m \geq 0$ satisfying $\frac{2\pi m}{\omega_0} \leq T$ and $\frac{2\pi m+\pi}{\omega_0} \leq T$, where $\omega_0 = \sqrt{\gamma A^{1/\gamma}}\, M$.

\item The velocity $v$ and its first-order spatial derivatives
remain uniformly bounded by $1$ at every point $z_j^{(2)}$
throughout $[0,T]$:
\[
 \begin{aligned}
 \bigl|v^{(i)}(z_j^{(2)},t)\bigr|&\leq1,\\
 \bigl|\partial_{x_k}v^{(i)}(z_j^{(2)},t)\bigr|&\leq1,
 \end{aligned}
\]
for all $i,k\in\{1,2,3\}$, $j\in\{1,\dots,N\}$, and
$t\in[0,T]$.
\end{enumerate}

Thus the construction combines large oscillations at the first
group of points with uniform bounds at the second group. This
produces a pronounced spatial contrast in both velocity and
velocity gradient while the solution remains smooth on
$[0,T]$. 

\begin{table}[htbp]
\centering
\small
\setlength{\tabcolsep}{3pt}
\renewcommand{\arraystretch}{1.3}
\begin{tabular}{@{}p{0.18\textwidth}
                p{0.30\textwidth}
                p{0.40\textwidth}@{}}
\toprule
\textbf{Aspect}
& \textbf{Classical singularity}
& \textbf{Quasi-singularity}
\tabularnewline
\midrule
\textbf{Mechanism}
& \raggedright Shock formation or implosion
& \raggedright Controlled oscillation and spatial contrast,
  produced through prescribed initial data, external forcing,
  and damping
\tabularnewline
\midrule
\textbf{Regularity}
& \raggedright A solution loses regularity in finite time
& \raggedright Each solution in the parameterized family
  $\{(p_M,v_M)\}$ remains smooth on the prescribed interval
  $[0,T]$
\tabularnewline
\midrule
\textbf{Amplitude}
& \raggedright Spatial derivatives may blow up in a shock,
  while physical variables may become unbounded in an implosion
& \raggedright At the first group of points, selected positive
  and negative peaks of $v_M$ and $\nabla v_M$ have magnitude
  greater than $M$
\tabularnewline
\midrule
\textbf{Spatial behavior}
& \raggedright Singular behavior may occur on a shock surface
  or at an implosion center
& \raggedright Large oscillations at the first group of points
  coexist with a bound by $1$ at the second group
\tabularnewline
\bottomrule
\end{tabular}
\caption{Comparison between classical singularity formation and
the quasi-singularity construction.}
\label{tab:singularity_vs_quasi}
\end{table}

The quasi-singularity construction provides a smooth setting in
which to study some quantitative features associated with singular
behavior. For each prescribed interval $[0,T]$, the family of
solutions $\{(p_M,v_M)\}$ exhibits increasingly large and rapid
velocity and gradient oscillations at the first group of points as
$M$ grows. At the same time, the pressure remains comparable to
the constant initial pressure
\[
 p_0=M^{\frac{2\gamma}{\gamma-1}},
 \qquad
 \|p_M(t)\|_{C_b^2}\leq Cp_0,
 \quad 0\leq t\leq T.
\]
The values at the second group of points remain uniformly bounded
by $1$, producing a spatial contrast between the prescribed
locations. Since the initial data, external force, and damping
parameter are specified explicitly, the construction also gives a
concrete framework for controlling this behavior. The main
features are summarized in Table~\ref{tab:singularity_vs_quasi}.

\subsection{Main results.}
Our main result constructs the quasi-singularity phenomenon for the isentropic compressible Euler equations.

\begin{theorem}[Quasi-singularity formation]
\label{thm:quasi-singularity}
Let $M \gg 1$ be an arbitrarily large parameter, $T>0$ be an arbitrary time, $N \ge 1$ be an arbitrary integer, and let $\{z^{(1)}_j\}_{j=1}^N, \{z^{(2)}_j\}_{j=1}^N \subset \mathbb{R}^3$ be two prescribed collections of pairwise distinct points.

Set the initial data $(p_{\mathrm{in}}, v_{\mathrm{in}})$ for the isentropic compressible Euler system \eqref{eq:euler_pressure_velocity} as
\begin{equation*}
  \begin{cases}
  p_{\mathrm{in}}(x)=p_0,\\
  v_{\mathrm{in}}(x)=v_h(x,0)=\nabla H_g(x),
  \end{cases}
\end{equation*}
where $p_0=M^{\frac{2\gamma}{\gamma-1}}$ as in \eqref{eq:p_0}, and the acoustic profile $v_h$ and Herglotz wave field $H_g$ are constructed in Corollary~\ref{thm:acoustic-profile}. The external force $f$ is defined by
\begin{equation*}
  f(x,t)=\kappa v_h(x,t),
\end{equation*}
and the damping parameter $\kappa$ is defined by
\begin{equation*}
  \kappa=K_*(1+T)M^2e^{K_*MT},
\end{equation*}
where the constant $K_* > 0$ depends only on $\gamma$, $A$, $N$, and the locations of the points $\{z^{(1)}_j\}_{j=1}^N$ and $\{z^{(2)}_j\}_{j=1}^N$, and is independent of $T$ and $M$.

Then the Cauchy problem for system \eqref{eq:euler_pressure_velocity} admits a unique classical solution $(p, v) \in C([0, T]; C^{2}(\mathbb{R}^3))$ satisfying the following properties:
\begin{enumerate}
\item[(i)] The pressure $p$ admits uniform lower and upper bounds proportional to $p_0$, together with a uniform $C_b^2$ bound; specifically, for all $(t,x) \in [0,T] \times \mathbb{R}^3$,
\[
 cp_0 \leq p(t,x) \leq Cp_0, \qquad
 \|p(t)\|_{C_b^2} \leq Cp_0,
\]
where constants $c,C$ depend only on $\gamma$, $A$, $N$, and the locations of the points $\{z^{(1)}_j\}_{j=1}^N$ and $\{z^{(2)}_j\}_{j=1}^N$, and is independent of $T$ and $M$.
\item[(ii)] Each component of the velocity $v=(v^{(1)},v^{(2)},v^{(3)})$ and its first-order spatial derivatives simultaneously exhibit localized, high-amplitude, high-frequency oscillatory behavior at the points $z^{(1)}_j$ ($j=1,\dots,N$) throughout $[0,T]$:
\begin{align}
v^{(i)}\!\left(z^{(1)}_j, \frac{2\pi m}{\omega_0}\right) &> M, \\
v^{(i)}\!\left(z^{(1)}_j, \frac{2\pi m+\pi}{\omega_0}\right) &< -M, \\
\partial_{x_k}v^{(i)}\!\left(z^{(1)}_j, \frac{2\pi m}{\omega_0}\right) &> M, \\
\partial_{x_k}v^{(i)}\!\left(z^{(1)}_j, \frac{2\pi m+\pi/}{\omega_0}\right) &< -M,
\end{align}
for all $i,k \in \{1,2,3\}$, $j \in \{1,\dots,N\}$, for all $i,k \in \{1,2,3\}$, $j \in \{1,\dots,N\}$, and every integer $m \geq 1$ such that  $ \frac{2\pi m}{\omega_0}$ and $ \frac{2\pi m+\pi}{\omega_0}$ lies in $[0,T]$, where $\omega_0 = \sqrt{\gamma A^{1/\gamma}}\, M$.
\item[(iii)] The velocity $v$ and its first-order spatial derivatives remain uniformly bounded at the points $z^{(2)}_j$ ($j=1,\dots,N$) throughout $[0,T]$:
\begin{align}
\bigl|v^{(i)}(z^{(2)}_j, t)\bigr| &\leq 1, \\
\bigl|\partial_{x_k}v^{(i)}(z^{(2)}_j, t)\bigr| &\leq 1,
\end{align}
for all $i,k \in \{1,2,3\}$, $j \in \{1,\dots,N\}$, and $t \in [0,T]$.
\end{enumerate}  
Thus, the behavior of $v$ and $\nabla v$ at $\{z^{(1)}_j\}_{j=1}^N$ and $\{z^{(2)}_j\}_{j=1}^N$ forms a sharp spatial contrast over $[0,T]$, combining intense high-frequency oscillations of amplitude greater than $M$ at $\{z^{(1)}_j\}_{j=1}^N$ with uniform boundedness by $1$ at $\{z^{(2)}_j\}_{j=1}^N$.
\end{theorem}

\begin{remark}
The initial pressure $p_0 = M^{\frac{2\gamma}{\gamma-1}}$ is chosen to ensure that the auxiliary pressure $p_*$ constructed in Lemma~\ref{lemma: auxiliary} admits uniform lower and upper bounds proportional to $p_0$. If $p_0$ were instead chosen as a constant independent of $M$, the lower bound of $p_*$ could decay exponentially in $M$ (of order $e^{-CM}$) and become arbitrarily small; this would cause severe difficulties in establishing the well-posedness of the remainder system for $(q,w)$ in Section~\ref{sec:remainder}, as the well-posedness relies on the uniform positive definiteness of the associated coefficient matrix.
 \end{remark}

\begin{remark}
Since the solution $(p,v)$ depends explicitly on the amplitude parameter $M$, the time $T$, and the geometric configuration $\{z^{(1)}_j\}_{j=1}^N \cup \{z^{(2)}_j\}_{j=1}^N$, our construction produces a parameterized family of globally smooth classical solutions $\{(p_M, v_M)\}$ on $[0,T]$ rather than a single solution blowing up in finite time. In particular, as $M$ grows, the family develops arbitrarily strong localized velocity and gradient oscillations ($>M$) at $\{z^{(1)}_j\}_{j=1}^N$ in sharp contrast to its uniform boundedness ($\leq 1$) at $\{z^{(2)}_j\}_{j=1}^N$, all while remaining globally smooth throughout $[0,T]$. This comes at the analytical price of requiring a sufficiently large damping parameter $\kappa$ and external force, which increases with $M$, $T$, and the number $N$ of prescribed points. 
\end{remark}

\begin{remark}
The mechanism underlying the construction of localized amplification and high-frequency oscillations in quasi-singularities is twofold: first, designing an explicit principal profile that generates the desired oscillatory concentration; and second, selecting a suitable damping parameter and external force to establish quantitative estimates and control over the remainder. Specifically, we adopt the ansatz
\begin{equation*}
    p = p_* + q, \qquad v = v_h + w,
\end{equation*}
where $(p_*, v_h)$ denotes the principal component: $v_h$ is obtained via the delicate construction of a Herglotz kernel and its associated Herglotz wave function, while $p_*$ solves a transport equation driven by $v_h$. The spatial localization and high-frequency oscillation properties originate from $v_h$—more explicitly, from the Herglotz wave function in space and the cosine modulation in time.

We then need to control the remainder term $w$: as long as $w$ remains sufficiently small throughout the prescribed time interval $[0,T]$ so as not to disrupt the oscillatory and localized behavior of $v_h$, the desired properties transfer directly to $v$. To ensure this, we suitably design the external force and choose the damping parameter $\kappa$ sufficiently large. For details on the construction of the initial data and the principal profile $(p_*, v_h)$, we refer to Section~\ref{sec:acoustic}; for the precise choice of $\kappa$ and the energy estimates showing how this damping parameter controls $(q,w)$, we refer to Section~\ref{sec:remainder}.
\end{remark}

\subsection{Notation}
$\mathrm{i}=\sqrt{-1}$.
For any integer $m \geq 0$, we denote by $C_b^m(\mathbb{R}^3)$ the space of $m$-times continuously differentiable functions on $\mathbb{R}^3$ whose derivatives up to order $m$ are bounded, equipped with the norm
\begin{equation*}
\|g\|_{C_b^m} := \sum_{|\alpha| \leq m} \|\partial_x^{\alpha} g\|_{L^{\infty}(\mathbb{R}^3)} = \sum_{|\alpha| \leq m} \sup_{x \in \mathbb{R}^3} |\partial_x^{\alpha} g(x)|,
\end{equation*}
where $\alpha = (\alpha_1, \alpha_2, \alpha_3) \in \mathbb{N}_0^3$ is a multi-index with $|\alpha| = \alpha_1 + \alpha_2 + \alpha_3$ and $\partial_x^{\alpha} = \partial_{x_1}^{\alpha_1} \partial_{x_2}^{\alpha_2} \partial_{x_3}^{\alpha_3}$.

For convenience, we adopt the following standard asymptotic notation throughout this paper. For positive quantities $x$ and $y$, we write
\begin{align*}
    x \lesssim y &\iff x \le C y, \\
    x \gtrsim y &\iff x \ge C y, \\
    x \simeq y &\iff C_1 y \le x \le C_2 y,
\end{align*}
where $C, C_1, C_2 > 0$ are generic positive constants depending only on the adiabatic exponent $\gamma$, the physical constant $A$, and the geometric configuration, but strictly independent of the parameter $M$ and $T$.

\subsection{Organization}
The rest of this paper is organized as follows. Section~\ref{sec:acoustic} is devoted to the construction of the initial data, external force, and principal profile $(p_*, v_h)$, as well as the selection of the damping parameter $\kappa$ for system~\eqref{eq:euler_pressure_velocity}. Section~\ref{sec:remainder} establishes quantitative energy estimates and uniform control for the remainder $(q,w)$. Section~\ref{eq:proof} completes the proof of Theorem~\ref{thm:quasi-singularity}. Section~\ref{sec:numerical-contrast} presents the numerical experiment and reports the velocity and gradient diagnostics for two damping coefficients.
%%%%%%%%%%%%%%%%%%%%%%%%%%%%%%%%%%%%%%%%%%%%%%%%%%
%%%%%%%%%%%%%%%%%%%%%%%%%%%%%%%%%%%%%%%%%%%%%%%%%%
\section{Construction of initial data, external force, and damping parameter for system~\eqref{eq:euler_pressure_velocity}}\label{sec:acoustic}
In this section, we detail the construction of the initial data $(v_{\mathrm{in}}, p_{\mathrm{in}})$, the external force $f$, and the damping parameter $\kappa$ for system \eqref{eq:euler_pressure_velocity}. To achieve this, we begin by constructing a Herglotz wave field.

In the following lemma, we construct an explicit Herglotz wave field whose value, as well as its first- and second-order derivatives, can be explicitly prescribed at a finite set of points. The construction relies on the structural properties of Herglotz wave functions and the far-field patterns of the fundamental solution to the Helmholtz equation, combined with Rellich's lemma and standard linear algebra arguments.

\begin{lemma}[Exact pointwise control of Herglotz waves]
\label{lem:herglotz-control}
Let $N\geq 1$, and let
\[
\bigl\{z_j \bigr\}_{j=1}^N\subset\mathbb R^3
\]
be $N$ pairwise distinct points. For $g\in L^2(\mathbb S^2)$, define the Herglotz wave function
\begin{equation}
\label{eq:herglotz-definition}
H_g(x):=\int_{\mathbb S^2}e^{\mathrm{i}x\cdot d}g(d)\,dS(d),
\qquad x\in\mathbb R^3,
\end{equation}
where $dS$ is the standard surface measure on $\mathbb S^2$.
For any prescribed block vector $C=(c_j)_{j=1}^N\in\mathbb C^{9N}$ with components
\begin{equation}\label{eq:prescribed block vector}
c_j=(c_{j,0},c_{j,1},\ldots,c_{j,8})\in\mathbb C^9,
\qquad j=1,\ldots,N,
\end{equation}
there exists an explicit Herglotz kernel 
\[
g := \mathcal L^* G^{-1} C,
\]
where $\mathcal L^*$ and $G$ are defined in \eqref{eq:L*-definition} and \eqref{eq:Gram-definition}, respectively, such that the corresponding Herglotz wave function $H_g$ satisfies 
\begin{equation}
\label{eq:general-interpolation}
\begin{aligned}
H_g(z_j) &= c_{j,0}, & \partial_{x_k} H_g(z_j) &= c_{j,k}, \quad k=1,2,3, \\
\partial_{x_1}^2 H_g(z_j) &= c_{j,4}, & \partial_{x_1}\partial_{x_2} H_g(z_j) &= c_{j,5}, \\
\partial_{x_1}\partial_{x_3} H_g(z_j) &= c_{j,6}, & \partial_{x_2}^2 H_g(z_j) &= c_{j,7}, \\
\partial_{x_2}\partial_{x_3} H_g(z_j) &= c_{j,8}, & & \qquad\qquad j=1,\ldots,N.
\end{aligned}
\end{equation}
Moreover, kernel $g$ satisfies
\begin{equation}
\label{eq:norm-bounds}
\frac{\|C\|_2}{\sqrt{\lambda_{\max}(G)}}
\leq \|g\|_{L^2(\mathbb S^2)}
\leq \frac{\|C\|_2}{\sqrt{\lambda_{\min}(G)}},
\end{equation}
where $||C||_2$ denotes the standard Euclidean norm on $\mathbb C^{9N}$, and $G$ is the positive definite matrix introduced below. Here, $\lambda_{\max}(G)$ and $\lambda_{\min}(G)$ denotes the maximum eigenvalue and minimum eigenvalue of $G$, respectively. These quantities depend only on $N$ and the point configuration, and are independent of $C$, where $C$ denotes the vector of prescribed values and derivatives at these points. Furthermore, if every $c_j$ is real,
then kernel $g$ satisfies $g(-d)=\overline{g(d)}$ and $H_g$ is
real-valued throughout $\mathbb R^3$.
\end{lemma}

\begin{proof} We divide the proof into 4 steps.

\textbf{Step 1: Multi-point evaluation operator and the far-field patterns.}
We use inner products that are linear in the first argument:
$$
\langle h,f\rangle_{L^2(\mathbb S^2)}
=\int_{\mathbb S^2}h(d)\overline{f(d)}\,dS(d),
\qquad
\langle A,B\rangle_{\mathbb C^{9N}}
=\sum_{j=1}^N\sum_{m=0}^8
A_{j,m}^{(r)}\overline{B_{j,m}^{(r)}}.
$$
To be consistent with the notation for $C\in\mathbb C^{9N}$ introduced above, we identify each vector $A,B\in\mathbb C^{9N}$ with the indexed collections of their components,$
A=\bigl(A_{j,m}^{(r)}\bigr)_{\substack{1\le j\le N\\\\0\le m\le 8}},
 B=\bigl(B_{j,m}^{(r)}\bigr)_{\substack{1\le j\le N\\\\0\le m\le 8}}.
$
For notational convenience, we introduce the following differential operators:
\begin{equation}
\label{eq:derivative-list}
(D_0,D_1,\ldots,D_8)
:=(1,\partial_{x_1},\partial_{x_2},\partial_{x_3},
\partial_{x_1}^2,\partial_{x_1}\partial_{x_2},
\partial_{x_1}\partial_{x_3},\partial_{x_2}^2,
\partial_{x_2}\partial_{x_3}).
\end{equation}
Here $D_0$ denotes the identical operator. $D_{m,z}$ denotes the same differential operator acting on $z$.
For $z\in\mathbb R^3$ and $d\in\mathbb S^2$, define
\begin{equation}
\label{eq:patterns}
\begin{aligned}
\phi_{z,0}(d)&:=e^{-\mathrm{i}z\cdot d},\\
(\phi_{z,1}(d),\ldots,\phi_{z,8}(d))
&:=e^{-\mathrm{i}z\cdot d}
\bigl(-\mathrm{i}d_1,-\mathrm{i}d_2,-\mathrm{i}d_3,-d_1^2,-d_1d_2,
-d_1d_3,-d_2^2,-d_2d_3\bigr).
\end{aligned}
\end{equation}
Thus $\phi_{z,m}=D_{m,z}e^{-\mathrm{i}z\cdot d}$ for $m=0,1,\ldots,8$.
Differentiation under the integral sign gives
\[
H_h(z)=D_0H_h(z)=\langle h,\phi_{z,0}\rangle,
\qquad
D_mH_h(z)=\langle h,\phi_{z,m}\rangle,
\quad m=1,\ldots,8.
\]
Define the bounded linear operator
\begin{equation}
\label{eq:L-definition}
\begin{aligned}
\mathcal L:L^2(\mathbb S^2)&\longrightarrow\mathbb C^{9N},\\
\mathcal Lh
&:=\bigl(D_mH_h(z_j)\bigr)_{j,m}
=\bigl(\langle h,\phi_{z_j,m}\rangle\bigr)_{j,m},
\end{aligned}
\end{equation}
where $j=1,\ldots,N$, and $m=0,1,\ldots,8$.
Its adjoint operator is
\begin{equation}
\label{eq:L*-definition}
\begin{aligned}
\mathcal L^*A:\mathbb C^{9N}&\longrightarrow L^2(\mathbb S^2) ,\\
\mathcal L^*A
&:
=\sum_{j=1}^N\sum_{m=0}^8
A_{j,m}\phi_{z_j,m}.
\end{aligned}
\end{equation}
Set
\begin{equation}
\label{eq:Gram-definition}
G=\mathcal L\mathcal L^*:\mathbb C^{9N}\longrightarrow \mathbb C^{9N},\qquad
G_{(j,m),(i,n)}
=\bigl\langle\phi_{z_i,n},\phi_{z_j,m}\bigr\rangle_{L^2(\mathbb S^2)},
\end{equation}
where $G_{(j,m),(i,n)}$ denotes the entry of $G$ in the row indexed by $(j,m)$ and the column indexed by $(i,n)$. $\phi_{z_i,n}$ and $\phi_{z_j,m}$ are defined in \eqref{eq:patterns}. We next introduce the outgoing fundamental solution of the Helmholtz equation,
\[
\Phi(x,z)=\frac{e^{\mathrm{i}|x-z|}}{4\pi|x-z|},\qquad x\ne z.
\]
With $R=|x|$ and $d=x/R$, direct expansion gives, uniformly in
$ d \in\mathbb S^2$,
\begin{equation}
\label{eq:far-field-expansions}
\begin{aligned}
\Phi(x,z)
&=\frac{e^{\mathrm{i}R}}{4\pi R}
\left(e^{-\mathrm{i} z\cdot d }+O(R^{-1})\right),\\
\partial_{z_k}\Phi(x,z)
&=\frac{e^{\mathrm{i}R}}{4\pi R}
\left(-\mathrm{i}d_k e^{-\mathrm{i}z\cdot d }+O(R^{-1})\right),\\
\partial_{z_k}\partial_{z_\ell}\Phi(x,z)
&=\frac{e^{\mathrm{i}R}}{4\pi R}
\left(-d_kd_\ell e^{-\mathrm{i}z\cdot d }
+O(R^{-1})\right).
\end{aligned}
\end{equation}
Recall that the far-field pattern $u_\infty$ of a radiating solution $u$ is defined through the asymptotic behavior
$$
u(x)=\frac{e^{\mathrm{i}R}}{R}\left(u_\infty(d)+O(R^{-1})\right),
\qquad R\to\infty.
$$
It therefore follows from \eqref{eq:far-field-expansions} that
$\phi_{z,0}$ and $\phi_{z,m}$ are precisely the far-field patterns of
$4\pi\Phi(\cdot,z)$ and $4\pi D_{m,z}\Phi(\cdot,z)$, respectively.

\medskip
\textbf{Step 2: Linear independence of far-field patterns and positive definiteness of $G$.}
For every $A\in\mathbb C^{9N}$,
\begin{equation}
\label{eq:Gram-quadratic-form}
A^*GA
=\langle\mathcal L\mathcal L^*A,A\rangle
=\|\mathcal L^*A\|_{L^2(\mathbb S^2)}^2,
\end{equation}
where $A^*$ denotes the conjugate transpose of $A$. Hence, in view of the representation of $\mathcal{L}^*$ in \eqref{eq:L*-definition}, $G$ is positive definite if and only if $\mathcal{L}^*$ is injective. This is equivalent to the linear independence of the family $\{\phi_{z_j,m}:j=1,\ldots,N,\ m=0,1,\ldots,8\}$.

Suppose that there exist coefficients $a_{j,m}$ ($j=1,\dots,N$, $m=0,\dots,8$) such that 
\begin{equation}
\label{eq:vanishing-combination}
  \sum_{j=1}^N \sum_{m=0}^8 a_{j,m} \phi_{z_j,m}(d) = 0, \quad \text{for all } d \in \mathbb{S}^2.
\end{equation}
Define
\begin{equation}
\label{eq:outgoing-field}
U(x):=4\pi\sum_{j=1}^N\sum_{m=0}^8
a_{j,m}
\left.D_{m,z}\Phi(x,z)\right|_{z=z_j},
\qquad x\in\mathbb R^3\setminus\mathcal Z,
\end{equation}
where $\mathcal Z:=\{z_j\}_{j=1}^N$.
Then $U$ is outgoing, $(\Delta+1)U=0$ in
$\mathbb R^3\setminus\mathcal Z$, and its far-field pattern is zero by
\eqref{eq:far-field-expansions} and \eqref{eq:vanishing-combination}.
In particular, $U(x)=O(|x|^{-2})$ as $|x|\to\infty$, and hence
\[
\lim_{R\to\infty}\int_{|x|=R}|U(x)|^2\,dS(x)=0.
\]
Let
\[
0<\varepsilon<\frac13
\min_{\substack{z,\widetilde z\in\mathcal Z\\z\ne\widetilde z}}
|z-\widetilde z|,
\qquad D_\varepsilon:=\bigcup_{z\in\mathcal Z}B_\varepsilon(z).
\]
The exterior $\mathbb R^3\setminus\overline{D_\varepsilon}$ is connected.
Rellich's lemma \cite{ColtonKress} in its connected exterior-domain form therefore yields
$U=0$ in $\mathbb R^3\setminus\overline{D_\varepsilon}$.
For each fixed $x\notin\mathcal Z$, we may additionally choose
$\varepsilon<\operatorname{dist}(x,\mathcal Z)$.
It follows that
\begin{equation}
\label{eq:U-zero}
U=0\qquad\text{in }\mathbb R^3\setminus\mathcal Z.
\end{equation}
Fix $z_0=z_{j_0}$ for some $j_0=1,\ldots,N$, and write
$x=z_0+\rho w$, where $\rho>0$ and $w\in\mathbb S^2$. Near $z_0$, we have
\begin{equation}
\label{eq:singular-expansions}
\begin{aligned}
\Phi(z_0+\rho w,z_0)&=\frac{e^{\mathrm{i}\rho}}{4\pi\rho},\\
\left.\partial_{z_k}\Phi(z_0+\rho w,z)\right|_{z=z_0}
&=\frac{e^{\mathrm{i}\rho}}{4\pi}w_k
\left(\frac1{\rho^2}-\frac{\mathrm{i}}{\rho}\right),\\
\left.\partial_{z_k}\partial_{z_\ell}
\Phi(z_0+\rho w,z)\right|_{z=z_0}
&=\frac{e^{\mathrm{i}\rho}}{4\pi}
\left[
\frac{3w_kw_\ell-\delta_{k\ell}}{\rho^3}
-\frac{\mathrm{i}(3w_kw_\ell-\delta_{k\ell})}{\rho^2}
-\frac{w_kw_\ell}{\rho}
\right].
\end{aligned}
\end{equation}
All terms in $U$ centered at points distinct from $z_0$ remain bounded as $\rho\downarrow0$. For the fixed point $z_0$, set
$a_m=a_{j_0,m}$. We now multiply \eqref{eq:U-zero} by $\rho^3$ and let $\rho\downarrow0$. By \eqref{eq:singular-expansions}, the contributions from $\Phi$ and its first-order derivatives vanish in this limit, while only the leading $\rho^{-3}$ singularities of the second-order derivatives remain. Recalling from  \eqref{eq:derivative-list} that $a_4,\ldots,a_8$ correspond, respectively, to
$\partial_{z_1}^2$, $\partial_{z_1}\partial_{z_2}$,
$\partial_{z_1}\partial_{z_3}$, $\partial_{z_2}^2$, and
$\partial_{z_2}\partial_{z_3}$, we obtain
\begin{equation}
\label{eq:quadratic-identity}
a_4(3w_1^2-1)+3a_5w_1w_2+3a_6w_1w_3
+a_7(3w_2^2-1)+3a_8w_2w_3=0,
\qquad w\in\mathbb S^2.
\end{equation}
Let $e_1=(1,0,0)$, $e_2=(0,1,0)$, and $e_3=(0,0,1)$. Taking $w=e_3$ and $w=e_1$ yields
\[
-a_4-a_7=0,\qquad 2a_4-a_7=0,
\]
so $a_4=a_7=0$.
Next, taking
$w=(e_1+e_2)/\sqrt2$, $(e_1+e_3)/\sqrt2$, and
$(e_2+e_3)/\sqrt2$ in \eqref{eq:quadratic-identity} gives
$a_5=a_6=a_8=0$.
Thus all five second-derivative coefficients at $z_0$ vanish.
Multiplying the remaining identity by $\rho^2$ and letting
$\rho\downarrow0$ gives
\[
a_1w_1+a_2w_2+a_3w_3=0,\qquad w\in\mathbb S^2,
\]
and hence $a_1=a_2=a_3=0$. Multiplying the remaining identity by $\rho$ and letting
$\rho\downarrow0$ gives $a_0=0$.
Since $z_0$ was arbitrary, all coefficients in
\eqref{eq:vanishing-combination} vanish.
The stated family is therefore linearly independent, and
\eqref{eq:Gram-quadratic-form} proves that $G$ is Hermitian positive
definite. In particular,
\begin{equation}
\label{eq:spectral-bounds}
0<\lambda_{\min}(G)I\leq G\leq\lambda_{\max}(G)I,
\end{equation}
where $\lambda_{\min}(G)$, $\lambda_{\max}(G)$ depends on $\bigl\{z_j\bigr\}_{j=1}^N$.

\medskip
\textbf{Step 3: Construction of the Herglotz kernel and its estimate.}
Set
\begin{equation}
\label{eq:density-construction}
\beta = (\beta_{j,m})_{\substack{1 \le j \le N \\ 0 \le m \le 8}}:=G^{-1}C,\qquad
g:=\mathcal L^*G^{-1}C
=\sum_{j=1}^N\sum_{m=0}^8
\beta_{j,m}\phi_{z_j,m},
\end{equation}
where $C$ is given by \eqref{eq:prescribed block vector}.
Then by \eqref{eq:Gram-definition}, we have
\[
\mathcal Lg=\mathcal L\mathcal L^*G^{-1}C=GG^{-1}C=C.
\]
By the definition of $\mathcal L$ given in \eqref{eq:L-definition}, we prove \eqref{eq:general-interpolation}. Using \eqref{eq:Gram-quadratic-form}, \eqref{eq:spectral-bounds}, \eqref{eq:density-construction}, we obtain the estimates for kernel $g$,
\begin{equation*}
 \frac{||C||^2_2}{\lambda_{\max}(G)} \leqslant \|g\|_{L^2(\mathbb S^2)}^2
=\beta^*G\beta=C^*G^{-1}C\leqslant \frac{||C||^2_2}{\lambda_{\min}(G)}.
\end{equation*}
This proves \eqref{eq:norm-bounds}.

\medskip
\textbf{Step 4: Real-valuedness of the Herglotz field for real prescribed data $C$.}
Every function in \eqref{eq:patterns} satisfies
\begin{equation}
\label{eq:pattern-symmetry}
\phi_{z,m}(-d)=\overline{\phi_{z,m}(d)}.
\end{equation}
Changing variables $d\mapsto-d$ in the Gram integrals therefore gives
\[
\overline{G_{(j,m),(i,n)}}
=G_{(j,m),(i,n)}.
\]
Thus $G$ is real symmetric positive definite.
If $C\in\mathbb R^{9N}$, then $\beta=G^{-1}C\in\mathbb R^{9N}$.
By \eqref{eq:density-construction} and \eqref{eq:pattern-symmetry},
$g(-d)=\overline{g(d)}$.
It follows that, for every $x \in \mathbb{R}^3$, taking the complex conjugate to $H_g(x)$ and changing variables yields
\[
\overline{H_g(x)}
=\int_{\mathbb S^2}e^{-\mathrm{i}x\cdot d}\overline{g(d)}\,dS(d)
=\int_{\mathbb S^2}e^{\mathrm{i}x\cdot d}\overline{g(-d)}\,dS(d)
=H_g(x).
\]
Hence $H_g$ is real-valued. 
\end{proof}

\begin{remark}
The above construction has two useful consequences. First, the values of $H_g$ and its derivatives up to second order can be arbitrarily prescribed at finitely many points, subject to the compatibility condition imposed by the Helmholtz equation.Since
$$
(\Delta+1)H_g=0,
$$
we have, at every point $z$,
$$
\partial_{x_1}^2H_g(z)
+\partial_{x_2}^2H_g(z)
+\partial_{x_3}^2H_g(z)
=-H_g(z).
$$
Therefore, $H_g(z)$ and the three pure second derivatives
$$
\partial_{x_1}^2H_g(z),\qquad
\partial_{x_2}^2H_g(z),\qquad
\partial_{x_3}^2H_g(z)
$$
cannot be prescribed independently. In particular, once $H_g(z)$ and two of these three second derivatives are prescribed, the remaining one is determined by the above identity. Second, the Herglotz kernel $g$ obtained in the proof is explicitly computable. More precisely, for any compatible prescribed data $C$ and prescribed points, the proof gives
$$
g=\mathcal L^*(\mathcal L\mathcal L^*)^{-1}C,
$$
where the operator $\mathcal L$, and hence the matrix $\mathcal L\mathcal L^*$, is completely determined by the positions of the prescribed points. Thus $g$ can be computed directly from the data $C$ and the locations of the prescribed points by solving a finite-dimensional linear system.
\end{remark}

By Lemma~\ref{lem:herglotz-control}, for any $M \gg 1$ and any two collections of distinct points $\{z^{(1)}_j\}_{j=1}^N$ and $\{z^{(2)}_j\}_{j=1}^N$, we can construct smooth solutions $(p_h, v_h)$ to the linearized compressible Euler equations exhibiting a sharp spatial contrast: at each $z^{(1)}_j$, every component of $v_h$ and its first-order spatial derivatives undergo large-amplitude, high-frequency oscillations with amplitude $M+1$ and frequency of order ${O}(M)$, whereas at each $z^{(2)}_j$, they remain uniformly small (bounded by $\frac{1}{2}$) for all $t \ge 0$. Moreover, $v_h$ and $p_h$, along with all of their spatial derivatives, are globally bounded by ${O}(M)$ and ${O}(M^{\frac{2\gamma}{\gamma-1}})$, respectively.

\begin{corollary}
\label{thm:acoustic-profile}
Let $M\gg 1$ be arbitrary, and let $\{z^{(1)}_j\}_{j=1}^N$ ,$\{z^{(2)}_j\}_{j=1}^N$ be arbitrary disjoint two collection of points. The acoustic profile $(p_h,v_h)$ has the following form:
\begin{equation}
\label{eq:smooth_pair}
\begin{cases}
\displaystyle
p_h(x,t):=\omega_{0} \rho_0 H_g(x) \sin(\omega_0t),
\\[2ex]
\displaystyle
v_h(x,t)
:=
\nabla H_g(x) \cos(\omega_0t).
\end{cases}
\end{equation}
Here 
\begin{equation}\label{eq:omega_0}
\rho_{0}:=\rho(p_0),\quad \omega_0:=\sqrt{\frac{\gamma p_0}{\rho(p_0)}}.
\end{equation}
$p_0$ is a constant depending on $M$ and $\gamma$,
\begin{equation}\label{eq:p_0}
p_0=M^{\frac{2\gamma}{\gamma-1}}.
\end{equation}
Then 
\begin{equation}\label{eq omega_0 rho_0}
\rho_0=A^{-\frac{1}{\gamma}}M^{\frac{2}{\gamma-1}},
\qquad
 \omega_0=\sqrt{\frac{\gamma p_0}{\rho_0}}=\sqrt{\gamma A^{1/\gamma}} M.
\end{equation}
The Herglotz function $H_g$ is defined in \eqref{eq:herglotz-definition} and Herglotz kernel $g$ is constructed by
\begin{equation}\label{eq:kernel g depending on M}
g:=\mathcal L^*G^{-1}C,
\end{equation}
where $\mathcal L^*$ and $G$ are given in
Lemma~\ref{lem:herglotz-control}. The vector $C = (c^{(1)}_j, c^{(2)}_j)_{j=1}^N \in \mathbb{R}^{18N}$ is given by
\begin{equation}
c^{(1)}_j = (-3M-3, M+1, M+1, M+1, M+1, M+1, M+1, M+1, M+1),
\label{eq: choice of C} 
\end{equation}
\begin{equation}
c^{(2)}_j = \bigl(-\tfrac{3}{2}, \tfrac{1}{2}, \tfrac{1}{2}, \tfrac{1}{2}, \tfrac{1}{2}, \tfrac{1}{2}, \tfrac{1}{2}, \tfrac{1}{2}, \tfrac{1}{2}\bigr),\label{eq: choice of C 2}
\end{equation}
for each $j = 1, \dots, N$.

Then the following statements hold.
\begin{enumerate}

\item
The pair $(p_h, v_h)$ is real-valued and satisfies the linear acoustic system
\begin{equation}
\label{eq:linear-acoustic}
\left\{
\begin{aligned}
\partial_t p_h + \gamma p_0 \nabla \cdot v_h &= 0, \\
\rho_0 \partial_t v_h + \nabla p_h &= 0,
\end{aligned}
\right.
\qquad
\text{in }\mathbb{R}^3\times[0,\infty).
\end{equation}

\item 
At each point $z^{(1)}_j$ ($j = 1, 2, \dots, N$), the components of $v_h$ and its first-order spatial derivatives exhibit large-amplitude, high-frequency oscillations: for $t\ge 0$,
\begin{equation}
\label{eq:acoustic-point-values}
\begin{aligned}
v_h(z^{(1)}_j, t)
&= (M+1, M+1, M+1)^\top \cos(\omega_0 t), \\
\nabla_x v_h(z^{(1)}_j, t)
&= \begin{pmatrix}
M+1 & M+1 & M+1 \\
M+1 & M+1 & M+1 \\
M+1 & M+1 & M+1
\end{pmatrix} \cos(\omega_0 t).
\end{aligned}
\end{equation}

\item
At each point $z^{(2)}_j$ ($j = 1, 2, \dots, N$), the components of $v_h$ and its first-order spatial derivatives remain uniformly small for all $t \ge 0$:
\begin{equation}\label{eq: zj_2 smalles}
|v^{(i)}_h(z^{(2)}_j, t)| \le \frac{1}{2}
\quad \text{and} \quad
|\partial_{x_k} v^{(i)}_h(z^{(2)}_j, t)| \le \frac{1}{2}
\qquad \text{for all } i, k \in \{1, 2, 3\}.
\end{equation}
Here, $v^{(i)}_h$ denotes the $i$-th component of $v_h$ ($i = 1, 2, 3$).

\item For every multi-index $\alpha \in \mathbb{N}_0^3$, the spatial derivatives $\partial_x^\alpha H_g$, $\partial_x^\alpha v_h$, and $\partial_x^\alpha p_h$ are uniformly bounded:
\begin{align}
\sup_{x\in\mathbb{R}^3} |\partial_x^\alpha H_g(x)| &\leq C M, \label{eq:H_g-global-estimates} \\
\sup_{(x,t)\in\mathbb{R}^3\times[0,\infty)} |\partial_x^\alpha p_h(x,t)| &\leq C M^{\frac{2\gamma}{\gamma-1}}, \label{eq:acoustic-global-estimates} \\
\sup_{(x,t)\in\mathbb{R}^3\times[0,\infty)} |\partial_x^\alpha v_h(x,t)| &\leq C M, \label{eq:acoustic-global-estimates-v}\\
\sup_{(x,t)\in\mathbb{R}^3\times[0,\infty)} |\partial_x^\alpha \partial_t v_h(x,t)| &\leq C M^2. \label{eq:acoustic-global-estimates-partial_t-v}
\end{align}
where the constant $C$ depends on $A$, $\gamma$, the locations of the points $z^{(1)}_j$ and $z^{(2)}_j$, and the total number of points $2N$, but is independent of $M$.
\end{enumerate}
\end{corollary}

\begin{proof}
By Lemma~\ref{lem:herglotz-control}, since $C = (c^{(1)}_j, c^{(2)}_j)_{j=1}^N$ is real, $H_g$ is real-valued. Combined with the fact that $\omega_0$ and $\rho_0$ are real, it follows that the pair $(v_h, p_h)$ defined in \eqref{eq:smooth_pair} is real-valued. By using the fact that $(\Delta + 1)H_g = 0$ in $\mathbb{R}^3$, together with \eqref{eq:smooth_pair} and \eqref{eq:omega_0}, a direct computation gives that \eqref{eq:linear-acoustic} holds.

To verify \eqref{eq:acoustic-point-values}, by using Lemma~\ref{lem:herglotz-control} and the choice of $C$ in \eqref{eq: choice of C}, we have, at each $z^{(1)}_j$ ($j = 1, 2, \dots, N$),
\[
H_g(z^{(1)}_j) = -3M - 3,
\qquad
\nabla H_g(z^{(1)}_j) = (M+1, M+1, M+1)^\top,
\]
and
\[
\partial_{x_k x_l}^2 H_g(z^{(1)}_j) = M+1,
\qquad (k,l) \in \{1,2,3\}^2 \setminus \{(3,3)\}.
\]
Since $(\Delta + 1)H_g = 0$ in $\mathbb{R}^3$, it follows that
\[
\partial_{x_3}^2 H_g(z^{(1)}_j)
=
-H_g(z^{(1)}_j) - \partial_{x_1}^2 H_g(z^{(1)}_j) - \partial_{x_2}^2 H_g(z^{(1)}_j)
= M+1.
\]
Combining the above identities with the expression for $v_h$ in \eqref{eq:smooth_pair}, we immediately obtain \eqref{eq:acoustic-point-values}.
 
Similarly, at each $z^{(2)}_j$ ($j = 1, 2, \dots, N$), by Lemma~\ref{lem:herglotz-control} and \eqref{eq: choice of C 2}, we have
\[
H_g(z^{(2)}_j) = -\frac{3}{2},
\qquad
\nabla H_g(z^{(2)}_j) = \Bigl(\frac{1}{2}, \frac{1}{2}, \frac{1}{2}\Bigr)^\top,
\]
and
\[
\partial_{x_k x_l}^2 H_g(z^{(1)}_j) =\frac{1}{2},
\qquad (k,l) \in \{1,2,3\}^2 .
\]
Combining the above identities with the bound $|\cos(\omega_0 t)| \le 1$ and the expression for $v_h$ in \eqref{eq:smooth_pair}, we obtain \eqref{eq: zj_2 smalles}.

Finally, for every multi-index $\alpha \in \mathbb{N}_0^3$, combining the expression for $p_h$ in \eqref{eq:smooth_pair}, the definitions in \eqref{eq omega_0 rho_0}, the estimate for the kernel $g$ given in \eqref{eq:norm-bounds} of Lemma~\ref{lem:herglotz-control}, the definition of the vector $C$ in \eqref{eq: choice of C}--\eqref{eq: choice of C 2}, and the Cauchy--Schwarz inequality, we obtain for $x \in \mathbb{R}^3$,
\begin{equation}\label{eq:estimate-for-H_g}
\begin{aligned}
|\partial_x^\alpha H_g(x)|
&= \biggl|\int_{\mathbb{S}^2} (\mathrm{i}d)^\alpha e^{\mathrm{i}x\cdot d} g(d)\,dS(d)\biggr| \\
&\le \sqrt{4\pi} \lVert g \rVert_{L^2(\mathbb{S}^2)} \\
&\le \sqrt{4\pi} \frac{\sqrt{17N\bigl((M+1)^2+\frac{1}{4}\bigr)}}{\sqrt{\lambda_{\min}(G)}} \\
&\le 17\sqrt{\frac{\pi N}{\lambda_{\min}(G)}}\, M,
\end{aligned}
\end{equation}
where, by Lemma~\ref{lem:herglotz-control}, $\lambda_{\min}(G)$ depends only on the locations of the points $z^{(1)}_j$ and $z^{(2)}_j$ ($j = 1, \dots, N$). Using \eqref{eq:smooth_pair}, \eqref{eq omega_0 rho_0}, and \eqref{eq:estimate-for-H_g}, we deduce that for all $t \ge 0$, $x \in \mathbb{R}^3$,
\begin{equation}\label{eq:estimate-for-p_h}
\begin{aligned}
|\partial_x^\alpha p_h(x,t)|
&= \bigl|\omega_0 \rho_0 \partial_x^\alpha H_g(x)\sin(\omega_0 t)\bigr| \\
&\le \sqrt{\gamma A^{1/\gamma}}\, M \cdot A^{-\frac{1}{\gamma}} M^{\frac{2}{\gamma-1}} |\partial_x^\alpha H_g(x)| \\
&\le 17\sqrt{\frac{\pi \gamma N}{A^{1/\gamma}\lambda_{\min}(G)}}\, M^{\frac{2\gamma}{\gamma-1}}.
\end{aligned}
\end{equation}
Similarly, for $v_h$ in \eqref{eq:smooth_pair}, applying \eqref{eq:estimate-for-H_g} with the multi-index $\alpha + e_i$ yields, for every $i \in \{1, 2, 3\}$ and all $t \ge 0$, $x \in \mathbb{R}^3$,
\begin{equation}\label{eq:estimate-for-v_h}
\begin{aligned}
|\partial_x^\alpha v^{(i)}_h(x,t)|
&= \bigl|\partial_x^\alpha \partial_{x_i} H_g(x)\cos(\omega_0 t)\bigr| \\
&\le |\partial_x^{\alpha + e_i} H_g(x)| \\
&\le 17\sqrt{\frac{\pi N}{\lambda_{\min}(G)}}\, M,
\end{aligned}
\end{equation}
where $e_i$ denotes the $i$-th standard basis vector in $\mathbb{R}^3$. For the last one,
\begin{equation}\label{eq:2.37}
\begin{aligned}
|\partial_x^\alpha \partial_t v^{(i)}_h(x,t)|
&= |\omega_0 \partial_x^\alpha \partial_{x_i} H_g(x)\sin(\omega_0 t)| \\
&\leq \sqrt{\gamma A^{1/\gamma}}\, M \cdot 17\sqrt{\frac{\pi N}{\lambda_{\min}(G)}}\, M \\
&= 17\sqrt{\frac{\pi N \gamma A^{1/\gamma}}{\lambda_{\min}(G)}}\, M^2.
\end{aligned}
\end{equation}
Combining \eqref{eq:estimate-for-H_g}-\eqref{eq:2.37}, we obtain \eqref{eq:H_g-global-estimates}-\eqref{eq:acoustic-global-estimates-partial_t-v}.
\end{proof}

To proceed, we apply Corollary~\ref{thm:acoustic-profile} to construct an auxiliary pressure $p_*$.

\begin{lemma}\label{lemma: auxiliary}
Let $p_*$ solve the linear transport equation for the pressure with the prescribed velocity field $v_h$:
\begin{equation}\label{eq:reference-continuity}
\partial_t p_* + v_h \cdot \nabla p_* + \gamma p_* \nabla \cdot v_h = 0,
\qquad p_*(0,x) = p_0,
\end{equation}
where $v_h$ is the acoustic velocity profile given in \eqref{eq:smooth_pair} and the initial data $p_0$ is defined in \eqref{eq:p_0}.Then the following estimates hold for any $t \ge 0$, $x \in \mathbb{R}^3$, and $m \in \mathbb{N}_0$:
\begin{equation}\label{eq:lower-and-upper-bound-for-p-star}
c p_0 \leq p_*(x,t) \leq C p_0,
\end{equation}
\begin{equation}\label{eq:higher-order}
\|p_*(t)\|_{C_b^m} \leq C_m p_0,
\end{equation}
\begin{equation}\label{eq:composite-estimate}
\|\rho(p_*(t))^{-1}\nabla p_*(t)\|_{C_b^m} \leq C_m M^2,
\end{equation}
\begin{equation}\label{eq:estimate-for-mathcalR}
\|\mathcal R(t)\|_{C_b^m} \leq C_m M^2,
\end{equation}
where 
\begin{equation}\label{eq:R}
\mathcal R := \partial_t v_h + v_h \cdot \nabla v_h + \rho(p_*)^{-1}\nabla p_*.
\end{equation}
Here, the positive constants $c, C, C_m$ depend on $A$, $\gamma$, the locations of the points $z^{(1)}_j, z^{(2)}_j$, and $N$ from Corollary~\ref{thm:acoustic-profile}, but are independent of $M$ and time $t$.
\end{lemma}
\begin{proof}
Throughout this proof, the constant $C$ changes line by line and is independent of $M$ and $t$.

We claim that $p^*$ satisfying equation \eqref{eq:reference-continuity} has the following explicit expression,
\begin{equation}\label{eq:expression of $p_*$}
\begin{aligned}
 p_*(t,x)
 &=p_0J_{b_M(t)}(Y_{-b_M(t)}(x))^{-\gamma}\\
 &=p_0J_{-b_M(t)}(x)^\gamma.
\end{aligned}
\end{equation}
Here, $Y_s$ denotes the flow generated by the gradient of the Herglotz wave field $\nabla H_g$ constructed in Corollary \ref{thm:acoustic-profile}:
\begin{equation}\label{eq:autonomous-flow}
\frac{d}{ds}Y_s(x) = \nabla H_g(Y_s(x)), \qquad Y_0(x) = x.
\end{equation}
Denote by $J_s(x)$ the determinant of $D_x Y_s(x)$:
\begin{equation}\label{eq:definition of J_s}
J_s(x) := \det D_x Y_s(x).
\end{equation}
Define 
\begin{equation}\label{eq:flow-parameter}
 b_M(t):=\frac{1}{\omega_0}\sin(\omega_0t)=\frac{1}{\sqrt{\gamma A^{1/\gamma}} M}\sin(\sqrt{\gamma A^{1/\gamma}} Mt),
\end{equation}
where $\omega_0$ is defined by \eqref{eq omega_0 rho_0}.

We now verify \eqref{eq:expression of $p_*$}. We calculate, by expression of $v_h$ given in \eqref{eq:smooth_pair} and chain rule,
\begin{equation}\label{eq:reference-characteristics}
\begin{aligned}
 \frac{d}{dt}Y_{b_M(t)}(x)
 &=b_M'(t)\nabla H_g(Y_{b_M(t)}(x))\\
 &=\cos(\omega_0t)\nabla H_g(Y_{b_M(t)}(x))\\
 &=v_h(t,Y_{b_M(t)}(x)).
\end{aligned}
\end{equation}
Differentiating \eqref{eq:autonomous-flow} with respect to the spatial variable $x$ yields
\begin{equation}\label{eq:flow-first-derivative}
\frac{d}{ds} D_x Y_s = D^2 H_g(Y_s) D_x Y_s, \qquad D_x Y_0 = I_3.
\end{equation}
Since $\nabla H_g$ is bounded and globally Lipschitz, the
ordinary differential equation \eqref{eq:autonomous-flow} has a unique solution for every
$s\in R$. Uniqueness gives the flow identities
\begin{equation}\label{eq:flow-group}
 Y_s\circ Y_r=Y_{s+r},\qquad Y_s^{-1}=Y_{-s}.
\end{equation}
The inverse flow in \eqref{eq:flow-group} implies that $D_xY_s$
is invertible. Using \eqref{eq:definition of J_s} and the determinant differentiation formula therefore gives
\[
\begin{aligned}
 \frac{d}{ds}J_s
 &=J_s\operatorname{tr}
   \bigl((D_xY_s)^{-1}D^2H_g(Y_s)D_xY_s\bigr)\\
 &=J_s\operatorname{tr}(D^2H_g(Y_s))
   =\Delta H_g(Y_s)J_s.
\end{aligned}
\]
Since $\Delta H_g=-H_g$ and $J_0(x)=1$, we obtain
\begin{equation}\label{eq:jacobian}
\begin{aligned}
 \frac{d}{ds}J_s(x)&=-H_g(Y_s(x))J_s(x),\\
 J_s(x)&=\exp\left(-\int_0^sH_g(Y_\tau(x))\,d\tau\right)>0.
\end{aligned}
\end{equation}
By equations~\eqref{eq:reference-continuity} and~\eqref{eq:reference-characteristics}, we have
\begin{equation}
\frac{d}{dt} p_*(t, Y_{b_M(t)}(x)) = -\gamma p_*(t, Y_{b_M(t)}(x)) (\nabla \cdot v_h)(t, Y_{b_M(t)}(x)).
\end{equation}
On the other hand, combining \eqref{eq:jacobian}, the expression for $v_h$ in \eqref{eq:smooth_pair}, and the identities $\Delta H_g = -H_g$ and $b_M'(t) = \cos(\omega_0 t)$ gives
\[
\frac{d}{dt} J_{b_M(t)}(x) = (\nabla \cdot v_h)(t, Y_{b_M(t)}(x)) J_{b_M(t)}(x).
\]
Consequently,
\[
 \frac{d}{dt}\left[
 p_*(t,Y_{b_M(t)}(x))J_{b_M(t)}(x)^\gamma
 \right]=0.
\]
Using the initial conditions $p_*(0,x)=p_0$ and $J_0=1$ yields
\begin{equation}\label{eq:reference-pressure}
 p_*(t,Y_{b_M(t)}(x))=p_0J_{b_M(t)}(x)^{-\gamma}.
\end{equation}
To obtain a formula at a fixed spatial point, by \eqref{eq:flow-group}, differentiate
$Y_b(Y_{-b}(x))=x$ and take determinants:
\begin{equation}\label{eq:inverse-jacobian}
 J_b(Y_{-b}(x))J_{-b}(x)=1.
\end{equation}
Then by \eqref{eq:flow-group}, \eqref{eq:reference-pressure} and \eqref{eq:inverse-jacobian}, for all $x \in \mathbb{R}^3$,
$$
p_*(t,x)=p_0J_{b_M(t)}(Y_{-b_M(t)}(x))^{-\gamma}=p_0J_{-b_M(t)}(x)^{\gamma}.
$$
This formula defines a smooth positive function for all time.
Direct calcualtion verifies \eqref{eq:reference-continuity}. Conversely, the same characteristic
calculation forces every solution with the prescribed initial data
to have this form.

We proceed to verify \eqref{eq:lower-and-upper-bound-for-p-star}. We observe that by \eqref{eq:flow-parameter},
\begin{equation}\label{eq:bounded-flow-parameter}
 |b_M(t)|\leq\frac{1}{\sqrt{\gamma A^{1/\gamma}} M}  \qquad(t\in R).
\end{equation}
This implies that the physical time $t$ can be arbitrarily large, the flow
parameter $s$ always stays in the fixed interval
$[-\frac{1}{\sqrt{\gamma A^{1/\gamma}} M},\frac{1}{\sqrt{\gamma A^{1/\gamma}} M}]$. 
For $|s|\leq \omega_0=\frac{1}{\sqrt{\gamma A^{1/\gamma}} M}$, formula~\eqref{eq:jacobian} and estimate \eqref{eq:H_g-global-estimates} give
\[
 |\log J_s(x)|
 \leq |s|\|H_g\|_{L^\infty}
 \leq{\frac{1}{\sqrt{\gamma A^{1/\gamma}} M}}\|H_g\|_{L^\infty} \leq {\frac{CM}{\sqrt{\gamma A^{1/\gamma}} M}} \leq \frac{C}{\sqrt{\gamma A^{1/\gamma}}}.
\]
Applying \eqref{eq:expression of $p_*$}, \eqref{eq:jacobian} and above inequalities yields
\begin{equation}\label{eq:reference-explicit-positivity}
 e^{-\frac{C}{\sqrt{\gamma A^{1/\gamma}}}}p_0
 \leq p_*(t,x)\\
 \leq e^{\frac{C}{\sqrt{\gamma A^{1/\gamma}}}}p_0.
\end{equation}
This holds for all $t \ge 0$ and $x \in \mathbb{R}^3$, which verifies \eqref{eq:lower-and-upper-bound-for-p-star}.

For estimate \eqref{eq:higher-order}, applying Gronwall's inequality to \eqref{eq:flow-first-derivative} 
and using estimate \eqref{eq:H_g-global-estimates} gives, for $|s|\leq \omega_0=\frac{1}{\sqrt{\gamma A^{1/\gamma}} M}$,
\begin{equation}\label{eq:flow-derivative-bound}
 \|D_xY_s\|_{L^\infty}
 \leq e^{|s|\|D^2H_g\|_{L^\infty}} \leq e^{\frac{C}{\sqrt{\gamma A^{1/\gamma}}}}.
\end{equation}
Differentiating \eqref{eq:flow-first-derivative} once more with respect to the spatial variables gives, for $i,j \in \{1,2,3\}$,
\begin{equation}\label{eq:flow-second-derivative}
\frac{d}{ds}\partial_{ij}Y_s
= D^2 H_g(Y_s)\partial_{ij}Y_s + D^3 H_g(Y_s)[\partial_i Y_s, \partial_j Y_s],
\qquad \partial_{ij}Y_0 = 0,
\end{equation}
where $\partial_i := \partial_{x_i}$ and $\partial_{ij} := \partial_{x_i x_j}$. Here, the second term arises from differentiating the composite factor $D^2 H_g(Y_s)$ via the chain rule: for each component $k \in \{1,2,3\}$,
\[
\bigl(D^3 H_g(Y_s)[\partial_i Y_s, \partial_j Y_s]\bigr)_k
:= \sum_{l,m=1}^3 \partial_{x_k x_l x_m}^3 H_g(Y_s)\, (\partial_i Y_s)_l \, (\partial_j Y_s)_m.
\]
By \eqref{eq:flow-derivative-bound}, the first-order derivatives $\partial_i Y_s$ and $\partial_j Y_s$ are uniformly bounded independently of $M$ on the interval $|s| \le \omega_0^{-1} = \frac{1}{\sqrt{\gamma A^{1/\gamma}}\, M}$, and hence so is the bilinear pair $[\partial_i Y_s, \partial_j Y_s]$. Applying Gr\"onwall's inequality to \eqref{eq:flow-second-derivative} together with the derivative estimates in \eqref{eq:H_g-global-estimates}, we obtain, for all $|s| \le \omega_0^{-1} = \frac{1}{\sqrt{\gamma A^{1/\gamma}}\, M}$,
\begin{equation}
\begin{aligned}
\lVert \partial_{ij} Y_s \rVert_{L^\infty}
&\le \exp\biggl(\biggl|\int_0^s \lVert D^2 H_g(Y_\tau) \rVert_{L^\infty}\, d\tau\biggr|\biggr)
     \biggl|\int_0^s \lVert D^3 H_g(Y_\tau)[\partial_i Y_\tau, \partial_j Y_\tau] \rVert_{L^\infty}\, d\tau\biggr| \\
&\le \exp\biggl(\frac{1}{\sqrt{\gamma A^{1/\gamma}}\, M} \cdot C M\biggr)
     \frac{1}{\sqrt{\gamma A^{1/\gamma}}\, M} \cdot C M \\
&\le C e^C,
\end{aligned}
\end{equation}
where $C > 0$ is independent of $M$. Consequently, $\partial_{ij} Y_s$ is uniformly bounded on $\bigl[-\frac{1}{\sqrt{\gamma A^{1/\gamma}}\, M}, \frac{1}{\sqrt{\gamma A^{1/\gamma}}\, M}\bigr]$ by a constant independent of $M$. At order $m\geq2$, the differentiated equation contains
$D^2H_g(Y_s)D_x^mY_s$ as its only term involving $D_x^mY_s$.
Every other term is a product of derivatives of $H_g$ and flow
derivatives of order at most $m-1$. Induction on $m$, followed by
Gronwall's inequality and estimates \eqref{eq:H_g-global-estimates}, proves
\begin{equation}\label{eq:flow-higher-derivatives}
 \sup_{|s|\leq \frac{1}{\sqrt{\gamma A^{1/\gamma}}\, M}}\|D_x^mY_s\|_{L^\infty}\leq C_m,
 \qquad m\geq1.
\end{equation}
Then, combining the bounds \eqref{eq:flow-higher-derivatives} and \eqref{eq:H_g-global-estimates}, the expression for $J_s(x)$ in \eqref{eq:jacobian}, and the chain rule, we obtain, for all $|s| \le \omega_0^{-1} = \frac{1}{\sqrt{\gamma A^{1/\gamma}}\, M}$,
\begin{equation}\label{eq:jacobian-uniform-spatial-bounds}
\sup_{|s|\leq \frac{1}{\sqrt{\gamma A^{1/\gamma}}\, M}}\|J_s\|_{C_b^m}\leq C_m.
\end{equation}
Finally, since $b_M(t)$ is independent of $x$, we have
\[
D_x^\alpha p_*(t,x) = p_0 \bigl. D_x^\alpha\bigl(J_s(x)^\gamma\bigr) \bigr|_{s=-b_M(t)}
\]
for every multi-index $\alpha \in \mathbb{N}_0^3$ with $|\alpha| \le m$. Using \eqref{eq:bounded-flow-parameter} together with the preceding bound \eqref{eq:jacobian-uniform-spatial-bounds} (and noting that $J_s$ is uniformly bounded away from zero, so $J_s^\gamma$ inherits the same $C_b^m$ bounds), we conclude that
\begin{equation*}
\sup_{t\geq0}\|p_*(t)\|_{C_b^m}\leq C_mp_0.
\end{equation*}

For estimate \eqref{eq:composite-estimate}, by \eqref{eq:pressure-density-relation} and \eqref{eq:omega_0}, we have
\[
 \rho(p_*)=\rho_0\left(\frac{p_*}{p_0}\right)^{1/\gamma}.
\]
Therefore
\begin{equation}\label{eq:reference-normalized-force}
 \rho(p_*)^{-1}\nabla p_*
 =\frac{p_0}{\rho_0}
   \left(\frac{p_*}{p_0}\right)^{-1/\gamma}
   \nabla\left(\frac{p_*}{p_0}\right).
\end{equation}
By previous estimates \eqref{eq:lower-and-upper-bound-for-p-star} and \eqref{eq:higher-order}, the ratio $p_*/p_0$ is uniformly bounded above and below by positive
constants, and all its fixed-order spatial derivatives of $p_*/p_0$ are uniformly
bounded. Consequently, combining the product and chain rules with \eqref{eq:lower-and-upper-bound-for-p-star}--\eqref{eq:higher-order}, \eqref{eq:p_0}, and \eqref{eq omega_0 rho_0} yields
\[
 \|\rho(p_*)^{-1}\nabla p_*\|_{C_b^m}
 \leq C_m\frac{p_0}{\rho_0}
 \leq C_mM^2.
\]

For estimate \eqref{eq:estimate-for-mathcalR}, by \eqref{eq:smooth_pair}, we have 
\[
\begin{aligned}
 \partial_t v_h
 &=-\omega_0 \sin(\omega_0t)\nabla H_g,\\
 (v_h\cdot\nabla)v_h
 &=\cos^2(\omega_0t)
           (\nabla H_g\cdot\nabla)\nabla H_g.
\end{aligned}
\]
Thus by estimates \eqref{eq omega_0 rho_0} and \eqref{eq:H_g-global-estimates}, we have
\[
 \|v_h(t)\|_{C_b^m}\leq C_mM,\qquad
 \|\partial_tv_h(t)\|_{C_b^m}
       +\|(v_h\cdot\nabla)v_h(t)\|_{C_b^m}\leq C_mM^2.
\]
Combining these bounds with
\eqref{eq:composite-estimate} proves \eqref{eq:estimate-for-mathcalR}. 
\end{proof}

Now, we are in a position to set the smooth initial data for compressible Euler system \eqref{eq:euler_pressure_velocity} as follows
  \begin{equation}\label{eq:initial data for compressible euler}
  \begin{cases}
  p_{\text{in}}(x)=p_0,\\
  v_{\text{in}}(x)=v_h(x,0)=\nabla H_g(x),
  \end{cases}
  \end{equation}
  where $p_0$ is defined in \eqref{eq:p_0}, $v_h$ and $H_g$ is constructed in Corollary \eqref{thm:acoustic-profile}. External force is defined by 
  \begin{equation}\label{eq:External force}
  f(x,t)=\kappa v_h(x,t). 
  \end{equation}
  Damping parameter $\kappa$ is defined by
  \begin{equation}\label{eq:kappa}
 \kappa=K_*(1+T)M^2e^{K_*MT},
  \end{equation}
 where $K_* > 0$ is a constant independent of $M$ and $T$ to be determined later, and the prescribed parameters $M$ and $T$ are given in Theorem~\ref{thm:quasi-singularity}.  
  %%%%%%%%%%%%%%%%%%%%%%%%%%%%%%%%%%%%%%%%%%%%%%%%
   %%%%%%%%%%%%%%%%%%%%%%%%%%%%%%%%%%%%%%%%%%%%%%%
 \section{Quantitative estimates and control of the remainder}\label{sec:remainder}

Assume that the classical solution $(p,v)$ to the compressible Euler system \eqref{eq:euler_pressure_velocity} admits the decomposition
\begin{equation}\label{eq:1.2}
p(x,t) = p_*(x,t) + q(x,t), \qquad 
v(x,t) = v_h(x,t) + w(x,t),
\end{equation}
where $p_*$ is the auxiliary pressure constructed in Lemma~\ref{lemma: auxiliary} and $v_h$ is the smooth linear acoustic velocity profile constructed in Corollary~\ref{thm:acoustic-profile}. $(q,w)$ is the unknown remainder. Since the initial data for $(p,v)$ in \eqref{eq:euler_pressure_velocity} are chosen according to \eqref{eq:initial data for compressible euler} and $p_*(x,0) = p_0$ by \eqref{eq:reference-continuity}, it follows that
\[
q(x,0) = 0, \qquad w(x,0) = 0, \qquad x \in \mathbb{R}^3.
\]
Substituting \eqref{eq:1.2} into the compressible Euler system \eqref{eq:euler_pressure_velocity} and using \eqref{eq:reference-continuity}, we obtain for the pressure (mass) equation:
\begin{equation}\label{eq:mass-equation}
\begin{aligned}
&\quad \partial_t q + (v_h+w)\cdot \nabla q + \gamma p \nabla \cdot w \\
&= -\partial_t p_* - v_h \cdot \nabla p_* - \gamma p_* \nabla \cdot v_h - w \cdot \nabla p_* - \gamma q \nabla \cdot v_h \\
&= -w \cdot \nabla p_* - \gamma q \nabla \cdot v_h.
\end{aligned}
\end{equation}
Similarly, for the momentum equation, using the external force defined in \eqref{eq:External force}, we have
\begin{equation}\label{eq:momentum-equation}
\begin{aligned}
&\quad \partial_t w + (v_h+w)\cdot \nabla w + \rho(p_*+q)^{-1} \nabla q + \kappa w \\
&= -\partial_t v_h - (v_h+w)\cdot \nabla v_h - \rho(p_*+q)^{-1}\nabla p_* \\
&= -\bigl(\partial_t v_h + v_h\cdot \nabla v_h + \rho(p_*)^{-1}\nabla p_*\bigr) - w \cdot \nabla v_h - \bigl(\rho(p_*+q)^{-1} - \rho(p_*)^{-1}\bigr)\nabla p_* \\
&=- w \cdot \nabla v_h - \bigl(\rho(p_*+q)^{-1} - \rho(p_*)^{-1}\bigr)\nabla p_*-\mathcal{R},
\end{aligned}
\end{equation}
where source term $\mathcal{R}$ is defined in \eqref{eq:R}.

Next, we show that for any prescribed large $M \gg 1$, any $T > 0$, and the prescribed points $\{z^{(1)}_j\}_{j=1}^N$ and $\{z^{(2)}_j\}_{j=1}^N$ from Corollary~\ref{thm:acoustic-profile}, one can choose the damping parameter $\kappa$ suitably so that the remainder $(q,w)$ exists on $[0,T]$ and satisfies $\Vert{}w(t)\Vert{}_{C_b^1} < \frac{1}{2}$ for all $t \in [0,T]$.

\begin{lemma}\label{thm:estimate and control of w}
For the fixed parameters and profiles above, choose the damping parameter $\kappa$ in \eqref{eq:momentum-equation} sufficiently large as
\[
\kappa = K_*(1+T)M^2 e^{K_* M T},
\]
where the constant $K_* > 0$ depends only on $\gamma$, $A$, $N$, and the locations of the points $\{z^{(1)}_j\}_{j=1}^N$ and $\{z^{(2)}_j\}_{j=1}^N$, and is independent of $T$ and $M$. Then the Cauchy problem~\eqref{eq:mass-equation}--\eqref{eq:momentum-equation} with zero initial data admits a unique classical solution $(q,w)$ on $[0,T] \times \mathbb{R}^3$ satisfying
\begin{equation}\label{eq:main-velocity-estimate}
\max_{0\leq t\leq T}\|w(t)\|_{C_b^1} \leq \frac{1}{2},
\end{equation}
\begin{equation}\label{eq:coarse-pressure}
\max_{0\leq t\leq T}\|q(t)\|_{C_b^2} \leq C p_0,
\end{equation}
where the constant $C > 0$ depends only on $\gamma$, $A$, $N$, and the locations of the points $\{z^{(1)}_j\}_{j=1}^N$ and $\{z^{(2)}_j\}_{j=1}^N$, and is independent of $T$ and $M$.
\end{lemma}

\begin{proof}
Throughout the proof, unless stated otherwise, the positive constants $C, C_m, C_E, c$ may vary from line to line; they depend only on $\gamma, A, N$, the locations of the points $\{z^{(1)}_j\}_{j=1}^N$ and $\{z^{(2)}_j\}_{j=1}^N$, and the fixed order of differentiation, and are independent of $M, t, T$, and the damping parameter $\kappa$. Moreover, for notational convenience, we retain the shorthand
\[
p = p_* + q, \qquad v = v_h + w
\]
as in \eqref{eq:1.2} throughout the entire proof.

\noindent\textbf{Step 1: Notation and local well-posedness for $(q, w)$.}

Set $U=(q,w)^{\mathsf T}$. The associated matrices for equations
\eqref{eq:mass-equation}--\eqref{eq:momentum-equation} are
\begin{equation}\label{eq:evolution-matrices}
 B_j=
 \begin{pmatrix}
  v_j&\gamma p\,e_j^{\mathsf T}\\
  \rho(p)^{-1}e_j&v_jI_3
 \end{pmatrix}.
\end{equation}
Multiplication by
\begin{equation}\label{eq:pressure-symmetrizer}
 \widetilde A_0(p)=
 \begin{pmatrix}(\gamma p)^{-1}&0\\0&\rho(p)I_3\end{pmatrix}
\end{equation}
gives the symmetric matrices
\begin{equation}\label{eq:principal-matrices}
 \widetilde A_j=\widetilde A_0B_j
 =
 \begin{pmatrix}
  v_j/(\gamma p)&e_j^{\mathsf T}\\
  e_j&\rho(p)v_jI_3
 \end{pmatrix}.
\end{equation}
Then \eqref{eq:mass-equation}--\eqref{eq:momentum-equation} can be rewritten as
\begin{equation}\label{eq:symmetric-comparison-system}
\begin{aligned}
 &\widetilde A_0\partial_tU
 +\sum_{j=1}^3\widetilde A_j\partial_jU
 +\kappa
 \begin{pmatrix}
  0\\
  \rho(p)w
 \end{pmatrix}\\
 &\quad=
 \begin{pmatrix}
 \displaystyle
 -\frac{w\cdot\nabla p_*+\gamma  q \nabla \cdot v_h}{\gamma p}
 \\[2ex]
 \displaystyle
 -\rho(p)\,w\cdot\nabla v_h
 -\rho(p)
  \bigl(\rho(p)^{-1}-\rho(p_*)^{-1}\bigr)\nabla p_*
 -\rho(p)\mathcal R
 \end{pmatrix}.
\end{aligned}
\end{equation}
Source term $\mathcal{R}$ is defined in \eqref{eq:R}.

Since the Herglotz wave function is smooth but does not decay at infinity, we work in the uniformly local $H^s$ space to obtain the local well-posedness for $(q, w)$. Therefore, we introduce the following equivalent weighted norm:
$$\|u\|_{m,y}^2=\sum_{|\alpha|\leq m}
\int_{\mathbb{R}^3}\vartheta_y|D^\alpha u|^2\,dx,
$$
where weighted function
\begin{equation}\label{eq:weight}
 \vartheta_y(x)=e^{-\sqrt{1+|x-y|^2}},\qquad
 |\nabla\vartheta_y|\leq\vartheta_y.
\end{equation}
These weighted norms, after taking the supremum over $y \in \mathbb{R}^3$, are
equivalent to the corresponding uniformly local Sobolev norms.
We have the following Sobolev inequality,
\[
 \|u\|_{C_b^2}\leq C\sup_y\|u\|_{4,y}.
\]

Next, we define
\begin{equation}\label{eq:weighted-energy}
\begin{aligned}
 E_y(t)
 &:=\frac1{p_0}\sum_{|\alpha|\leq4}
       \int_{\mathbb{R}^3}\vartheta_y
       \left(\frac{|D^\alpha q(t,x)|^2}{\gamma p(t,x)}
                    +\rho(p(t,x))|D^\alpha w(t,x)|^2\right)\,dx,\\
 D_y(t)
 &:=\frac1{p_0}\sum_{|\alpha|\leq4}
       \int_{\mathbb{R}^3}\vartheta_y\rho(p(t,x))|D^\alpha w(t,x)|^2\,dx,\\
 \mathcal E(t)&:=\sup_{y \in \mathbb{R}^3} E_y(t),\\
 X_y(t)^2&:=\| q(t)/p_0\|_{4,y}^2+\|w(t)/M\|_{4,y}^2,
 \qquad \mathcal X(t):=\sup_{y \in \mathbb{R}^3}X_y(t).
\end{aligned}
\end{equation}
Here $p_0$ is the initial data for $p_*$ and $p$ and defined in \eqref{eq:p_0}.
 Since
$q(0)=w(0)=0$, we have $\mathcal X(0)=0$. For any $0<\delta\leq 1$, by the local well-posedness theory for symmetric hyperbolic systems \cite{Majda}, there exists a short time $\tau>0$ such that the classical solution $(q,w)$ exists on $[0,\tau]$ and satisfies
\begin{equation}\label{eq:local-wellposedness}
\mathcal{X}(t) < \delta
\end{equation}
for all $0 \leq t \leq \tau$. Choosing $\delta > 0$ sufficiently small so that $p$ satisfies uniform upper and lower bounds on this small time interval and combining \eqref{eq:omega_0}, \eqref{eq:p_0}, \eqref{eq:lower-and-upper-bound-for-p-star}, and \eqref{eq:higher-order} with the Sobolev embedding inequality, we obtain
\begin{equation}\label{eq:bootstrap-bounds}
\begin{gathered}
 \|q(t)\|_{C_b^2} \leq C \delta p_0,\\
 c p_0 \leq p(t,x) = p_*(t,x) + {q}(t,x) \leq C p_0,\\
 \|p(t)\|_{C_b^2} \leq C p_0,\qquad
 \|v(t)\|_{C_b^2} \leq C M,\\
 \rho(p(t,x)) \simeq \rho_0,\qquad
 \frac{\rho_0}{p_0} = \frac{\gamma}{c_0^2} \simeq M^{-2},
\end{gathered}
\end{equation}
which hold for every $(t,x) \in [0,\tau] \times \mathbb{R}^3$. Consequently,
\begin{equation}\label{eq:norm-equivalence}
 E_y(t) \simeq X_y(t)^2,\qquad
 D_y(t) \simeq \|w(t)/M\|_{4,y}^2,
 \qquad 0 \leq t \leq \tau.
\end{equation}

\noindent \textbf{Step 2:  Weighted energy estimates for higher order term involving $q$ and $w$.}

Denote
\begin{equation}\label{eq:lower-order}
\begin{aligned}
 S_{ q}&=-w\cdot\nabla p_*-\gamma q \nabla \cdot v_h,\\
 S_w&=-w\cdot\nabla v_h
       -\bigl(\rho(p)^{-1}-\rho(p_*)^{-1}\bigr)\nabla p_*.
\end{aligned}
\end{equation}
Next, set $U_\alpha=D^\alpha_x U$ and define
\begin{equation}\label{eq:commutator-components}
\begin{aligned}
 C_{\alpha,q}
 &=-\sum_j\left(
 [D^\alpha,v_j]\partial_j q
       +\gamma[D^\alpha,p]\partial_jw_j\right),\\
 C_{\alpha,w}
 &=-\sum_j\left(
 [D^\alpha,v_j]\partial_jw
       +[D^\alpha,\rho(p)^{-1}]\partial_j q\,e_j\right),\\
 F_\alpha&:=\begin{pmatrix}F_{\alpha, q}\\
              F_{\alpha,w}  \end{pmatrix}=
 \begin{pmatrix}D^\alpha S_{q}+C_{\alpha,q}\\
                 D^\alpha S_w+C_{\alpha,w}\end{pmatrix},
\end{aligned}
\end{equation}
where $\{e_1, e_2, e_3\}$ denotes the standard orthonormal basis of $\mathbb{R}^3$. For $|\alpha|\leq4$, applying $D^\alpha_x$ to \eqref{eq:mass-equation}--\eqref{eq:momentum-equation} and multiplying by $\widetilde A_0$, we obtain the differentiated symmetric system for $U_{\alpha}$:
\begin{equation}\label{eq:differentiated-system}
\begin{aligned}
 \widetilde A_0\partial_tU_\alpha
 +\sum_j\widetilde A_j\partial_jU_\alpha
 +\kappa\begin{pmatrix}0\\\rho(p)D^\alpha w\end{pmatrix}
 =
 \widetilde A_0F_\alpha
 -\begin{pmatrix}0\\\rho(p)D^\alpha\mathcal R\end{pmatrix},
\end{aligned}
\end{equation}
where $\mathcal R$ is defined in \eqref{eq:R}.

We claim the following estimates hold. Each estimate holds at every fixed time
$t\in[0,\tau]$, uniformly in $y\in R^3$.
\begin{align}
 \frac{\|D^\alpha S_{ q}(t)\|_{0,y}}{p_0}
 &\leq C\bigl(\|w(t)\|_{4,y}
                  +M\| q(t)/p_0\|_{4,y}\bigr)
 \leq CM X_y(t),\label{eq:scalar-source}\\
 \frac{\|D^\alpha(w\cdot\nabla v_h)(t)\|_{0,y}}M
 &\leq CM\|w(t)/M\|_{4,y}\leq CM X_y(t).
 \label{eq:transport-source}
\end{align}
\begin{equation}\label{eq:pressure-source}
 \frac1M
 \|D^\alpha((\rho(p)^{-1}-\rho(p_*)^{-1})\nabla p_*)(t)\|_{0,y}\leq
 \frac{Cp_0}{M\rho_0}\| q(t)/p_0\|_{4,y}
 \leq CM X_y(t).
\end{equation}
\begin{equation}\label{eq:commutator-bound}
\frac{\|([D^\alpha,v_j]\partial_j q)(t)\|_{0,y}}{p_0}
 +\frac{\|([D^\alpha,p]\partial_jw_j)(t)\|_{0,y}}{p_0}
       \leq CM X_y(t),\\
 \end{equation}
 \begin{equation}\label{eq:commutator-bound1}
 \frac{\|([D^\alpha,v_j]\partial_jw)(t)\|_{0,y}}M
 +\frac{\|([D^\alpha,\rho(p)^{-1}]
                         \partial_j q)(t)\|_{0,y}}M
       \leq CM X_y(t).
 \end{equation}
 \begin{equation}\label{eq:weighted-calculus}
 \sum_{|\alpha|\leq4}
 \left(
 \frac{\|F_{\alpha, q}(t)\|_{0,y}}{p_0}
 +\frac{\|F_{\alpha,w}(t)\|_{0,y}}M
 \right)\leq CM X_y(t).
\end{equation}
Here $X_y(t)$ is defined in \eqref{eq:weighted-energy}.

The estimates in \eqref{eq:scalar-source} and \eqref{eq:transport-source} follow directly from the bounds for $p_*$ in \eqref{eq:higher-order} and $v_h$ in \eqref{eq:acoustic-global-estimates-v}, together with the definition of $X_y(t)$ in \eqref{eq:weighted-energy}.

For estimate \eqref{eq:pressure-source}, by \eqref{eq:pressure-density-relation}, we have
\begin{equation}\label{eq:density-difference}
\begin{aligned}
 \rho(p)^{-1}-\rho(p_*)^{-1}
 =-\frac{A^{1/\gamma}}{\gamma} q
  \int_0^1(p_*+s q)^{-1-1/\gamma}\,ds.
\end{aligned}
\end{equation}
We now estimate this difference in the weighted $H^4$ norm.
For brevity, set
\[
 a:=\frac{ q}{p_0},
 \qquad
 P_s:=\frac{p_*+s q}{p_0}=\frac{p_*}{p_0}+sa,
 \qquad 0\leq s\leq1.
\]
The bounds for $p_*$ in
Lemma~\ref{lemma: auxiliary} and the bounds
\eqref{eq:bootstrap-bounds} imply
\[
 0<c\leq P_s(t,x)\leq C,
 \qquad  \left\|\frac{p_*(t)}{p_0}\right\|_{C_b^4}\leq C, \quad
 0\leq s\leq1,\quad
 (t,x)\in[0,\tau]\times\mathbb R^3,
\]
and
\begin{equation}\label{eq: estimate a and p_s}
 \|a(t)\|_{C_b^2}
 \leq C\delta\leq C,\quad \|P_s(t)\|_{C_b^2} \leq C,
 \qquad 0\leq t\leq\tau.
\end{equation}
Using $\rho_0^{-1}=A^{1/\gamma}p_0^{-1/\gamma}$, we rewrite
\eqref{eq:density-difference} as
\begin{equation}\label{eq: rewrite p^{-1}-p_{*}^{-1}}
 \rho(p)^{-1}-\rho(p_*)^{-1}
 =-\frac{\rho_0^{-1}}{\gamma}
 a\int_0^1P_s^{-1-1/\gamma}\,ds.
\end{equation}
Thus it remains to estimate spatial derivatives of $aP_s^{-1-1/\gamma}$ up to order four. For a fourth-order spatial derivative, Leibniz's rule distributes
the derivatives between $a$ and $P_s^{-1-1/\gamma}$. If at least
two derivatives fall on $a$, then at most two derivatives fall
on the second factor. The latter is bounded in $L^\infty$ by the
bounds on $P_s$, $p_*/p_0$, and $a$. We therefore put the
derivative of $a$ in the weighted $L^2$ norm. If at most one derivative falls on $a$, derivatives of order three
or four may fall on $P_s$. Since
\[
 D^jP_s
 =D^j(p_*/p_0)+sD^ja,
\]
we estimate $D^j(p_*/p_0)$ in $L^\infty$. When a term contains
$D^3a$ or $D^4a$, we put that factor in the weighted $L^2$ norm
and put the remaining factors, which have at most two derivatives,
in $L^\infty$. More explicitly, the product and chain rules give the pointwise
bound
\[
\begin{aligned}
 \sum_{|\alpha|\leq4}
 \bigl|D^\alpha(aP_s^{-1-1/\gamma})\bigr|
 \leq{}&
 C\sum_{j=0}^4|\nabla^j a|\\
 &+C
 \left(\sum_{j=0}^2|\nabla^j a|\right)
 \left(\sum_{j=3}^4|\nabla^j a|\right).
\end{aligned}
\]
Here the constant is independent of $s$, $t$, $y$, and $M$.
Using $\|a(t)\|_{C_b^2}\leq C\delta\leq C$ and the definition
of $\|\cdot\|_{4,y}$ in \eqref{eq:weighted-energy}, we obtain
\[
 \bigl\|a(t)P_s(t)^{-1-1/\gamma}\bigr\|_{4,y}
 \leq C\|a(t)\|_{4,y},
 \qquad 0\leq s\leq1,\quad 0\leq t\leq\tau.
\]
Integrating above inequality in $s$ and using \eqref{eq: rewrite p^{-1}-p_{*}^{-1}} gives 
\begin{equation}\label{eqeq}
 ||\rho(p)^{-1}-\rho(p_*)^{-1}||_{4,y}\leq \frac{C}{\rho_0}|| q/p_0||_{4,y}.
\end{equation}
Then estimate \eqref{eq:pressure-source} follows by \eqref{eq:bootstrap-bounds}, \eqref{eq:higher-order} and \eqref{eqeq}, 
\begin{equation*}
\frac1M
 \|D^\alpha((\rho(p)^{-1}-\rho(p_*)^{-1})\nabla p_*)(t)\|_{0,y}\leq
 \frac{Cp_0}{M\rho_0}\| q(t)/p_0\|_{4,y}
 \leq CM X_y(t).
\end{equation*}

For commutator estimate \eqref{eq:commutator-bound}, by definition, we have
\begin{equation}\label{eq: commutator for [D^alpha,v_j]partial_jwidehat q}
[D^\alpha,v_j]\partial_j q=\sum_{0<\beta\leq\alpha}
 \binom{\alpha}{\beta}
 D^\beta v_j\,D^{\alpha-\beta}\partial_ju.
\end{equation}
First suppose that $1\leq|\beta|\leq2$. We put $D^\beta v_j$
in $L^\infty$ and the other factor in the weighted $L^2$ norm.
By~\eqref{eq:bootstrap-bounds},
\[
 \left\|
 D^\beta v_j\,
 D^{\alpha-\beta}\partial_j q
 \right\|_{0,y}
 \leq CM\|q(t)\|_{4,y}.
\]
If $3\leq|\beta|\leq4$, we use $v_j=v_{h,j}+w_j$ and write
\[
 D^\beta v_j\,
 D^{\alpha-\beta}\partial_j q
 ={}
 D^\beta v_{h,j}\,
 D^{\alpha-\beta}\partial_j q+
 D^\beta w_j\,
 D^{\alpha-\beta}\partial_j q.
\]
For the first term, $D^\beta v_{h,j}$ is bounded in $L^\infty$
by~\eqref{eq:acoustic-global-estimates-v}. For the second term,
we put $D^\beta w_j$ in weighted $L^2$. Since $ |\alpha-\beta|+1\leq2$,
the derivative of $q$ is bounded in $L^\infty$ by \eqref{eq:bootstrap-bounds}.
Consequently,
\[
\left\|
 D^\beta v_j\,
 D^{\alpha-\beta}\partial_j q
 \right\|_{0,y} \leq
 CM\| q(t)\|_{4,y}
 +Cp_0\|w(t)\|_{4,y}.
\]
Combining both cases with~\eqref{eq: commutator for [D^alpha,v_j]partial_jwidehat q} gives
\begin{equation}\label{eq:commutator-v-q}
 \frac{1}{p_0}\|([D^\alpha,v_j]\partial_j q)(t)\|_{0,y}\\
 \leq
 C\left(
 M\| q(t)/p_0\|_{4,y}
 +M\|w(t)/M\|_{4,y} \right)\leq CM X_y(t).
\end{equation}
The terms
$[D^\alpha,p]\partial_jw_j$ and
$[D^\alpha,v_j]\partial_jw$ in \eqref{eq:commutator-bound} can be estimated in the similar way.

For estimate \eqref{eq:commutator-bound1}, we only need to estimate the term $[D^\alpha,\rho(p)^{-1}]\partial_j q$. We split
$ \rho(p)^{-1}
 =
 \bigl(\rho(p)^{-1}-\rho(p_*)^{-1}\bigr)
 +\rho(p_*)^{-1}.
$
Then,
\begin{equation}
\begin{aligned}
 [D^\alpha,\rho(p)^{-1}]\partial_j  q&=
 [D^\alpha,\rho(p)^{-1}-\rho(p_*)^{-1}] \partial_j q+
 [D^\alpha,\rho(p_*)^{-1}]\partial_j q\\
 &=\sum_{0<\beta\leq\alpha}
 \binom{\alpha}{\beta}
 D^\beta (\rho(p)^{-1}-\rho(p_*)^{-1})\,D^{\alpha-\beta}\partial_j q\\
 &\quad+\sum_{0<\beta\leq\alpha}
 \binom{\alpha}{\beta}
 D^\beta (\rho(p_*)^{-1})\,D^{\alpha-\beta}\partial_j q.
 \end{aligned}
\end{equation}
We first estimate the term containing $\rho(p)^{-1}-\rho(p_*)^{-1}$.
From \eqref{eq: estimate a and p_s} and \eqref{eq: rewrite p^{-1}-p_{*}^{-1}}, we can obtain
\begin{equation}\label{eq:density-difference-C2}
 \|\rho(p(t))^{-1}-\rho(p_*(t))^{-1}\|_{C_b^2}
 \leq \frac{C}{\rho_0}.
\end{equation}
For $1\leq|\beta|\leq2$, we put $D^\beta (\rho(p)^{-1}-\rho(p_*)^{-1})$ in
$L^\infty$ and the derivative of $q$ in weighted $L^2$:
\[
 \left\|
 D^\beta\bigl(\rho(p)^{-1}-\rho(p_*)^{-1}\bigr)
 D^{\alpha-\beta}\partial_j q
 \right\|_{0,y}\leq
 \frac{C}{\rho_0}\| q(t)\|_{4,y}=
 \frac{Cp_0}{\rho_0}
 \left\|\frac{ q(t)}{p_0}\right\|_{4,y}.
\]
For $3\leq|\beta|\leq4$, we put $D^\beta (\rho(p)^{-1}-\rho(p_*)^{-1})$ in weighted $L^2$ and the derivative of $ q$
in $L^\infty$. Using~\eqref{eqeq}
and~\eqref{eq:bootstrap-bounds}, we get
\[
\begin{aligned}
 &\left\|
 D^\beta\bigl(\rho(p)^{-1}-\rho(p_*)^{-1}\bigr)
 D^{\alpha-\beta}\partial_j q
 \right\|_{0,y}\\
 &\quad\leq
 \|\rho(p(t))^{-1}-\rho(p_*(t))^{-1}\|_{4,y}
 \| q(t)\|_{C_b^2}\\
 &\quad\leq
 \frac{Cp_0}{\rho_0}
 \left\|\frac{ q(t)}{p_0}\right\|_{4,y}.
\end{aligned}
\]
Thus
\begin{equation}\label{eq:commutator-density-difference}
\bigl\|
 ([D^\alpha,\rho(p)^{-1}-\rho(p_*)^{-1}]
 \partial_j q)(t)
 \bigr\|_{0,y}\leq
 \frac{Cp_0}{\rho_0}
 \left\|\frac{ q(t)}{p_0}\right\|_{4,y}.
\end{equation}
For $[D^\alpha,\rho(p_*)^{-1}]\partial_j q$, estimate \eqref{eq:higher-order} yields
\[
 \|D^\beta(\rho(p_*(t))^{-1})\|_{L^\infty}
 \leq \frac{C}{\rho_0},
 \qquad 1\leq|\beta|\leq4.
\]
Then 
\begin{equation}\label{eq:commutator-density-reference}
 \|([D^\alpha,\rho(p_*)^{-1}]
 \partial_j q)(t)\|_{0,y}\leq
 \frac{C}{\rho_0}\| q(t)\|_{4,y}=
 \frac{Cp_0}{\rho_0}
 \left\|\frac{ q(t)}{p_0}\right\|_{4,y}.
\end{equation}
Combining~\eqref{eq:bootstrap-bounds}, \eqref{eq:commutator-density-difference}
and~\eqref{eq:commutator-density-reference}, we conclude that
\begin{equation*}
\frac{1}{M}\|([D^\alpha,\rho(p)^{-1}]
 \partial_j q)(t)\|_{0,y}\leq
 \frac{Cp_0}{M\rho_0}
 \left\|\frac{ q(t)}{p_0}\right\|_{4,y}\leq CMX_y(t).
\end{equation*}

For estimate \eqref{eq:weighted-calculus}, it follows by estimates \eqref{eq: estimate a and p_s}-\eqref{eq:commutator-bound1} together with definitions \eqref{eq:lower-order} and \eqref{eq:commutator-components}.

\noindent\textbf{Step 3: Extending the local existence time of $(q, w)$ to the prescribed time $T$ and deriving coarse bounds on $[0,T]$.}

Multiplying~\eqref{eq:differentiated-system} by $2\vartheta_y U_\alpha^{\mathsf{T}}/p_0$, summing over $\alpha$, integrating by parts, and using symmetry together with the definitions of $E_y(t)$ and $D_y(t)$ in~\eqref{eq:weighted-energy}, we obtain
\begin{equation}\label{eq:energy-identity}
\begin{aligned}
 E_y'(t)+2\kappa D_y(t)
 ={}&
 \frac1{p_0}\sum_\alpha\int\vartheta_y
 U_\alpha^{\mathsf T}
 \left(\partial_t\widetilde A_0+
                \sum_j\partial_j\widetilde A_j\right)U_\alpha\,dx\\
 &+\frac1{p_0}\sum_{\alpha,j}\int(\partial_j\vartheta_y)
               U_\alpha^{\mathsf T}\widetilde A_jU_\alpha\,dx\\
 &+\frac2{p_0}\sum_\alpha\int\vartheta_y
               U_\alpha^{\mathsf T}\widetilde A_0F_\alpha\,dx\\
 &-\frac2{p_0}\sum_\alpha\int\vartheta_y\rho(p)
                        D^\alpha w\cdot D^\alpha\mathcal R\,dx.
\end{aligned}
\end{equation}

For the first term in the right side of \eqref{eq:energy-identity}, by the pressure equation in \eqref{eq:euler_pressure_velocity} and expression of $\rho$ in \eqref{eq:pressure-density-relation}, we have 
\[
 \partial_t\rho(p)+  \nabla \cdot(\rho(p)v)=0,
\]
and
\begin{align*}
 \partial_t\frac1{\gamma p}
 + \nabla \cdot\frac v{\gamma p}
 &=-\frac{p_t+v\cdot\nabla p}{\gamma p^2}
                  +\frac{  \nabla \cdot v}{\gamma p}\\
 &=\frac{\gamma+1}{\gamma p}  \nabla \cdot v.
\end{align*}
By \eqref{eq:pressure-symmetrizer} and \eqref{eq:principal-matrices} and above calculations, we obtain
\begin{equation}\label{eq:coefficient-divergence}
 \partial_t\widetilde A_0+\sum_j\partial_j\widetilde A_j
 =
 \begin{pmatrix}
  (\gamma+1)(  \nabla \cdot v)/(\gamma p)&0\\0&0
 \end{pmatrix}.
\end{equation}
Then using \eqref{eq:coefficient-divergence}, the bound for $v$ in \eqref{eq:bootstrap-bounds}, we obtain 
\begin{equation}\label{eq:first-term-estimate}
\begin{aligned}
\biggl|\frac{1}{p_0}\sum_{|\alpha|\leq 4}\int \vartheta_y
  U_\alpha^{\mathsf{T}}
  \left(\partial_t\widetilde{A}_0 + \sum_j\partial_j\widetilde{A}_j\right)U_\alpha\,dx\biggr|
&\leq \frac{CM}{p_0}\sum_{|\alpha|\leq 4} \int \vartheta_y \frac{|D^{\alpha}q|^2}{\gamma p}\,dx \\
&\leq CM E_y(t),
\end{aligned}
\end{equation}
for $0\leq t\leq\tau$.

For the second term on the right-hand side of~\eqref{eq:energy-identity}, using the form of $\widetilde{A}_j$ in~\eqref{eq:principal-matrices}, we obtain
\begin{equation}\label{eq:second-term-estimate}
\begin{aligned}
 \sum_j (\partial_j \vartheta_y) U_\alpha^{\mathsf{T}} \widetilde{A}_j U_\alpha
 &= (v \cdot \nabla\vartheta_y) \left(\frac{|D^\alpha {q}|^2}{\gamma p} + \rho(p)|D^\alpha w|^2\right) + 2 D^\alpha {q}\, \nabla\vartheta_y \cdot D^\alpha w \\
 &\leq C\left(|v| + \sqrt{\frac{\gamma p}{\rho(p)}}\right) \left(\frac{|D^\alpha {q}|^2}{\gamma p} + \rho(p)|D^\alpha w|^2\right) \vartheta_y \\
 &\leq C M \left(\frac{|D^\alpha {q}|^2}{\gamma p} + \rho(p)|D^\alpha w|^2\right) \vartheta_y,
\end{aligned}
\end{equation}
where we have used the Young inequality:
\[
 2|D^\alpha {q}|\,|D^\alpha w|
 \leq \sqrt{\frac{\gamma p}{\rho(p)}} \left(\frac{|D^\alpha {q}|^2}{\gamma p} + \rho(p)|D^\alpha w|^2\right),
\]
together with $|\nabla\vartheta_y| \leq \vartheta_y$ and $|v| + \sqrt{\gamma p/\rho(p)} \leq CM$ in \eqref{eq:bootstrap-bounds}. Integrating and summing over $|\alpha|\leq4$, we obtain
\begin{equation}\label{eq:weight-derivative-bound}
 \left|\frac1{p_0}\sum_{\alpha,j}\int
 (\partial_j\vartheta_y)
 U_\alpha^{\mathsf T}\widetilde A_jU_\alpha\,dx\right|
 \leq CM E_y(t),\qquad 0\leq t\leq\tau.
\end{equation}

For the third term on the right-hand side of~\eqref{eq:energy-identity}, using~\eqref{eq:pressure-symmetrizer}, \eqref{eq:bootstrap-bounds}, the expression for $F_{\alpha}$ in~\eqref{eq:commutator-components}, and the definition of $E_y(t)$ in~\eqref{eq:weighted-energy}, we deduce from the Cauchy--Schwarz inequality that
\begin{equation}\label{eq:third-term}
\begin{aligned}
 \biggl|\frac{2}{p_0}\int\vartheta_y
        U_\alpha^{\mathsf{T}}\widetilde{A}_0 F_\alpha\,dx\biggr|
 &= \biggl|\frac{2}{p_0}\int \vartheta_y \left(\frac{D^{\alpha} {q}}{\gamma p} F_{\alpha, {q}} + \rho(p)D^{\alpha}w \cdot F_{\alpha,w}\right)dx\biggr| \\
 &\leq C E_y(t)^{1/2}
       \left(\frac{\|F_{\alpha, {q}}(t)\|_{0,y}}{p_0}
                   + \frac{\|F_{\alpha,w}(t)\|_{0,y}}{M}\right) \\
 &\leq CM E_y(t)^{1/2} X_y(t) \\
 &\leq CM E_y(t),
\end{aligned}
\end{equation}
where in the last step we have used $X_y(t) \simeq E_y(t)^{1/2}$.

For the last term on the right-hand side of~\eqref{eq:energy-identity}, using the definition of $D_y(t)$ in~\eqref{eq:weighted-energy}, Young's inequality, the bounds for $\mathcal{R}$ in~\eqref{eq:estimate-for-mathcalR}, and the relation $\rho_0/p_0 \simeq M^{-2}$ from~\eqref{eq:bootstrap-bounds}, we obtain
\begin{equation}\label{eq:residual-absorption}
\begin{aligned}
 \frac{2}{p_0}\biggl|
 \sum_\alpha\int\vartheta_y\rho(p)
                    D^\alpha w\cdot D^\alpha\mathcal{R}\,dx\biggr|
 &\leq \kappa D_y(t)
       + \frac{1}{\kappa p_0}\sum_\alpha
                   \int\vartheta_y\rho(p)|D^\alpha\mathcal{R}|^2\,dx \\
 &\leq \kappa D_y(t) + \frac{C\rho_0 M^4}{\kappa p_0}
       \leq \kappa D_y(t) + \frac{CM^2}{\kappa},
\end{aligned}
\end{equation}
where we have also used the fact that $\sup_{y \in \mathbb{R}^3} \int \vartheta_y\,dx < \infty$ by \eqref{eq:weight}.

Combining~\eqref{eq:energy-identity}--\eqref{eq:residual-absorption}, we obtain
\begin{equation}\label{eq:energy-differential}
 E_y'(t) + \kappa D_y(t) \leq CM E_y(t) + \frac{CM^2}{\kappa},
 \qquad 0 \leq t \leq \tau,
\end{equation}
where the constant $C > 0$ is independent of $\kappa$, $y$, $M$, $t$, and $\tau$. For each fixed $y$, applying Gronwall's inequality to~\eqref{eq:energy-differential} together with $E_y(0)=0$ yields
\[
 E_y(t) \leq \frac{CM^2}{\kappa} \int_0^t e^{CM(t-s)}\,ds
        \leq \frac{CM^2 t}{\kappa} e^{CMt},
 \qquad 0 \leq t \leq \tau.
\]Taking the supremum over $y$ in this estimate yields
\begin{equation}\label{eq:energy-bound}
 \mathcal E(t)\leq\frac{CM^2t}{\kappa}e^{CMt}.
\end{equation}
The estimate holds for every $t\in[0,\tau]$.The norm equivalence~\eqref{eq:norm-equivalence} gives $\mathcal{X}(t)^2 \leq C_E \mathcal{E}(t)$ for a fixed constant $C_E > 0$. Combining this with the choice of the damping parameter
\[
\kappa = K_*(1+T)M^2 e^{K_* M T},
\]
we deduce that for every $0 \leq t \leq \tau$,
\begin{equation}\label{eq:bootstrap-improvement}
 \mathcal{X}(t)^2
 \leq \frac{C_E C}{K_*} \frac{t}{1+T} e^{(C-K_*)Mt}
 \leq \frac{\delta^2}{4},
\end{equation}
provided that $K_*$ is chosen sufficiently large (depending on $\delta$, but independent of $t$, $\tau$, $T$, and $M$).

Thus, on the entire interval $[0,\tau]$, $\mathcal{X}(t)\leq\delta$ we previously obtained in \eqref{eq:local-wellposedness} improves to $\mathcal{X}(t)\leq\delta/2$. By the local continuation
theory and the continuity of $\mathcal{X}(t)$, there exists a larger time $\tau_1 \in (\tau, T]$ such that the solution $( {q},w)$ extends to $[0,\tau_1]$ and satisfies $\mathcal{X}(t)\leq\delta$ for all $t \in [0,\tau_1]$. Repeating the above argument on $[0,\tau_1]$ again yields $\mathcal{X}(t)\leq\delta/2$ on $[0,\tau_1]$. Iterating this continuation procedure, we conclude that $( {q},w)$ extends to the whole interval $[0,T]$ with 
$$
\mathcal{X}(t)\leq\delta/2
$$  
for all $t \in [0,T]$. Furthermore, from the definition of $\mathcal{X}(t)$ and the Sobolev embedding inequality, we obtain, for all $(t,x)\in[0,T]\times R^3$,
\begin{equation}\label{eq:bounds-on-0T}
\begin{gathered}
 \| {q}(t)\|_{C_b^2}\leq C\delta p_0,\\
 cp_0\leq p(t,x)=p_*(t,x)+ {q}(t,x)\leq Cp_0,\\
 \|p(t)\|_{C_b^2}\leq Cp_0,\qquad
 \|v(t)\|_{C_b^2}\leq CM,\\
 \rho(p(t,x))\simeq\rho_0,\qquad
 \frac{\rho_0}{p_0}=\frac{\gamma}{c_0^2}\simeq M^{-2}.
\end{gathered}
\end{equation}
Consequently, for every $0\leq t\leq T$, it also follows that
\begin{equation}\label{eq:pressure-acceleration}
\begin{aligned}
 \|\rho(p(t))^{-1}\nabla p(t)\|_{C_b^1}
 &\leq C\left(\frac{p_0}{\rho_0} + \frac{p_0^2}{p_0\rho_0}\right)\\
 &\leq CM^2.
\end{aligned}
\end{equation}

\noindent\textbf{Step 4: Obtaining the final \texorpdfstring{$C^1$}{C1} estimate for \texorpdfstring{$w$}{w} on $[0,T]$.}

Returning to the full momentum equation
\eqref{eq:euler_pressure_velocity} and using $w=v-v_h$, we obtain the
following transport equation for $w$:
\begin{equation}\label{eq:velocity-ode}
 \partial_t w+v\cdot\nabla w+\kappa w=R_0,
 \qquad
 R_0:=-\partial_t v_h-v\cdot\nabla v_h-\rho(p)^{-1}\nabla p.
\end{equation}
Combining the estimates in \eqref{eq:bounds-on-0T} and \eqref{eq:pressure-acceleration} with those for $v_h$ in \eqref{eq:acoustic-global-estimates-v} and \eqref{eq:acoustic-global-estimates-partial_t-v}, we obtain
\begin{equation}\label{eq:step8-previous-bounds}
\begin{aligned}
\|R_0(t)\|_{C_b^1} &\leq \|\partial_t v_h(t)\|_{C_b^1}
 + \|v(t)\cdot\nabla v_h(t)\|_{C_b^1} + \|\rho(p(t))^{-1}\nabla p(t)\|_{C_b^1} \\
&\leq CM^2, \qquad 0\leq t\leq T.
\end{aligned}
\end{equation}
In particular, for every $0\leq t\leq T$, we have the following
coarse bound for $w$:
\begin{equation}\label{eq:step8-coarse-w}
 \|w(t)\|_{C_b^1}
 \leq \|v(t)\|_{C_b^1}+\|v_h(t)\|_{C_b^1}
 \leq CM.
\end{equation}
In the subsequent analysis, we use the damping term $\kappa$ to improve this to a
bound independent of $M$.

For $\ell=1,2,3$, differentiating \eqref{eq:velocity-ode} with respect to the spatial variables yields
\begin{equation}\label{eq:step8-differentiated}
 \partial_t(\partial_\ell w)
 +v\cdot\nabla(\partial_\ell w)+\kappa\partial_\ell w=R_\ell,
 \qquad
 R_\ell:=\partial_\ell R_0
 -(\partial_\ell v\cdot\nabla)w.
\end{equation}
Using \eqref{eq:step8-previous-bounds} and \eqref{eq:step8-coarse-w}, we have 
\begin{equation}\label{eq:step8-source-bound}
\begin{aligned}
 \|R_\ell(t)\|_{L^\infty}
 &\leq\|\partial_\ell R_0(t)\|_{L^\infty}
 +\|\partial_\ell v(t)\|_{L^\infty}
   \|\nabla w(t)\|_{L^\infty}\\
 &\leq CM^2,\qquad \ell=1,2,3,\quad 0\leq t\leq T.
\end{aligned}
\end{equation}

Next, using the method of characteristics, we derive explicit representations for $w$ and $w_{\ell}$ from \eqref{eq:velocity-ode} and \eqref{eq:step8-differentiated}, respectively, and thereby obtain a sharper $C^1$ estimate for $w$.

For arbitrary $a\in\mathbb R^3$, we introduce $\xi(s;a)$ be the characteristic of $v$:
\[
 \frac{d}{ds}\xi(s;a)=v(s,\xi(s;a)),\qquad \xi(0;a)=a.
\]
The bounds on $v$ and $\nabla v$ ensure that these characteristics
exist on $[0,T]$. For each fixed $t$, the map $a\mapsto\xi(t;a)$ is
onto $\mathbb R^3$: given $x$, solve the same characteristic equation
backwards from $\xi(t)=x$ to find its initial point $a$.

Along a characteristic, by chain rule and~\eqref{eq:velocity-ode}, \eqref{eq:step8-differentiated} , we obtain
\begin{equation}\label{eq:derivative of w}
 \frac{d}{ds}\left(e^{\kappa s}w(s,\xi(s;a))\right)
 =e^{\kappa s}R_0(s,\xi(s;a)),
\end{equation}
and
\begin{equation}\label{eq:derivative of w ell}
 \frac{d}{ds}\left(e^{\kappa s}\partial_{\ell}w(s,\xi(s;a))\right)
 =e^{\kappa s}R_0(s,\xi(s;a)).
 \end{equation}
Since $w(0,\cdot)=0$ and $\partial_\ell w(0,\cdot)=0$, integration over $[0,t]$ gives
\begin{align}
 w(t,\xi(t;a))
 &=\int_0^t e^{-\kappa(t-s)}R_0(s,\xi(s;a))\,ds,\\
 \partial_\ell w(t,\xi(t;a))
 &=\int_0^t e^{-\kappa(t-s)}R_\ell(s,\xi(s;a))\,ds.
\end{align}
Taking the supremum over $a\in \mathbb R^3$ and using the surjectivity of the characteristic flow, together with the bounds on $R_0$ and $R_{\ell} $ in \eqref{eq:step8-previous-bounds} and \eqref{eq:step8-source-bound}, we obtain
\begin{equation}\label{eq:refined-w}
 \begin{aligned}
 \|w(t)\|_{C_b^1}
 &\leq CM^2\int_0^t e^{-\kappa(t-s)}\,ds\\
 &\leq \frac{CM^2}{\kappa},\qquad 0\leq t\leq T.
 \end{aligned}
\end{equation}
Here, the constant $C$ is independent of $M$, $\kappa$, $t$, and $T$. By the choice of $ \kappa=K_*(1+T)M^2e^{K_*MT}$ in \eqref{eq:kappa} and choose $K_*$ appropriatly, we obtain the refined $C^1$ estimate for $w$:
\[
\max_{0 \leq t \leq T} \|w(t)\|_{C_b^1} < 1/2.
\]

This completes the proof of Lemma \ref{thm:estimate and control of w}.
\end{proof}

\section{Proof of Theorem \ref{thm:quasi-singularity}}\label{eq:proof}
\begin{proof}
Set the initial data $(p_{\mathrm{in}}, v_{\mathrm{in}})$ as in \eqref{eq:initial data for compressible euler}, let $v_h$ be the acoustic profile constructed in \eqref{eq:smooth_pair} of Corollary~\ref{thm:acoustic-profile}, and take the external force $f$ as in \eqref{eq:External force}. Moreover, let $p_*$ be the auxiliary pressure constructed in Lemma~\ref{lemma: auxiliary} satisfying \eqref{eq:reference-continuity}, and let $(q,w)$ solve the perturbation system \eqref{eq:mass-equation}--\eqref{eq:momentum-equation} with zero initial data. A direct calculation then shows that
\[
p = p_* + q, \qquad v = v_h + w
\]
indeed solve the original isentropic compressible Euler system \eqref{eq:euler_pressure_velocity} with the prescribed initial data and external force.

To verify property \textup{(i)}, this follows directly from \eqref{eq:bounds-on-0T} in Lemma~\ref{thm:estimate and control of w}.

To verify property \textup{(ii)}, by \eqref{eq:acoustic-point-values} in Corollary~\ref{thm:acoustic-profile}, at each point $z^{(1)}_j$ ($j=1,\dots,N$) and for all $i,k \in \{1,2,3\}$, we have
\begin{align*}
v^{(i)}_h\!\left(z^{(1)}_j, \frac{2\pi m}{\omega_0}\right) &= (M+1)\cos(2\pi m) = M+1, \\
v^{(i)}_h\!\left(z^{(1)}_j, \frac{2\pi m+\pi}{\omega_0}\right) &= (M+1)\cos(2\pi m+\pi) = -M-1, \\
\partial_{x_k}v^{(i)}_h\!\left(z^{(1)}_j, \frac{2\pi m}{\omega_0}\right) &= (M+1)\cos(2\pi m) = M+1, \\
\partial_{x_k}v^{(i)}_h\!\left(z^{(1)}_j, \frac{2\pi m+\pi}{\omega_0}\right) &= (M+1)\cos(2\pi m+\pi) = -M-1,
\end{align*}
for every integer $m \geq 0$ satisfying $\frac{2\pi m}{\omega_0} \leq T$ and $\frac{2\pi m+\pi}{\omega_0} \leq T$. On the other hand, by Lemma~\ref{thm:estimate and control of w} and the choice of the damping parameter
\[
\kappa = K_*(1+T)M^2 e^{K_* M T}
\]
with $K_* > 0$ sufficiently large, we have
\[
\max_{0\leq t\leq T}\|w(t)\|_{C_b^1} \leq \frac{1}{2}.
\]
Consequently, combining $v = v_h + w$ with the above bounds gives
\begin{align*}
v^{(i)}\!\left(z^{(1)}_j, \frac{2\pi m}{\omega_0}\right)
&= v^{(i)}_h\!\left(z^{(1)}_j, \frac{2\pi m}{\omega_0}\right) + w^{(i)}\!\left(z^{(1)}_j, \frac{2\pi m}{\omega_0}\right) \\
&> M+1-\frac{1}{2} > M, \\[4pt]
v^{(i)}\!\left(z^{(1)}_j, \frac{2\pi m+\pi}{\omega_0}\right)
&= v^{(i)}_h\!\left(z^{(1)}_j, \frac{2\pi m+\pi}{\omega_0}\right) + w^{(i)}\!\left(z^{(1)}_j, \frac{2\pi m+\pi}{\omega_0}\right) \\
&< -M-1+\frac{1}{2} < -M, \\[4pt]
\partial_{x_k}v^{(i)}\!\left(z^{(1)}_j, \frac{2\pi m}{\omega_0}\right)
&= \partial_{x_k}v^{(i)}_h\!\left(z^{(1)}_j, \frac{2\pi m}{\omega_0}\right) + \partial_{x_k}w^{(i)}\!\left(z^{(1)}_j, \frac{2\pi m}{\omega_0}\right) \\
&> M+1-\frac{1}{2} > M, \\[4pt]
\partial_{x_k}v^{(i)}\!\left(z^{(1)}_j, \frac{2\pi m+\pi}{\omega_0}\right)
&= \partial_{x_k}v^{(i)}_h\!\left(z^{(1)}_j, \frac{2\pi m+\pi}{\omega_0}\right) + \partial_{x_k}w^{(i)}\!\left(z^{(1)}_j, \frac{2\pi m+\pi}{\omega_0}\right) \\
&< -M-1+\frac{1}{2} < -M,
\end{align*}
for all $i,k \in \{1,2,3\}$, $j \in \{1,\dots,N\}$, and every integer $m \geq 0$ satisfying $\frac{2\pi m}{\omega_0} \leq T$ and $\frac{2\pi m+\pi}{\omega_0} \leq T$.

To verify property \textup{(iii)}, by the estimates for $v_h$ at $z^{(2)}_j$ ($j=1,\dots,N$) in \eqref{eq: zj_2 smalles} of Corollary~\ref{thm:acoustic-profile}, we have
\[
\bigl|v^{(i)}_h(z^{(2)}_j, t)\bigr| \leq \frac{1}{2}
\qquad \text{and} \qquad
\bigl|\partial_{x_k} v^{(i)}_h(z^{(2)}_j, t)\bigr| \leq \frac{1}{2}
\]
for all $i, k \in \{1, 2, 3\}$ and all $t \geq 0$. Combining these bounds with the estimate $\max_{0 \leq t \leq T} \|w(t)\|_{C_b^1} \leq \frac{1}{2}$ via $v = v_h + w$, we obtain
\begin{align*}
\bigl|v^{(i)}(z^{(2)}_j, t)\bigr|
&\leq \bigl|v^{(i)}_h(z^{(2)}_j, t)\bigr| + \bigl|w^{(i)}(z^{(2)}_j, t)\bigr|
\leq \frac{1}{2} + \frac{1}{2} = 1, \\
\bigl|\partial_{x_k} v^{(i)}(z^{(2)}_j, t)\bigr|
&\leq \bigl|\partial_{x_k} v^{(i)}_h(z^{(2)}_j, t)\bigr| + \bigl|\partial_{x_k} w^{(i)}(z^{(2)}_j, t)\bigr|
\leq \frac{1}{2} + \frac{1}{2} = 1,
\end{align*}
for all $i,k \in \{1,2,3\}$, $j \in \{1,\dots,N\}$, and $t \in [0,T]$.

\end{proof}

%%%%%%%%%%%%%%%%%%%%%%%%%%%%%%%%%%%%%%%%%%%%%%%%%%
%%%%%%%%%%%%%%%%%%%%%%%%%%%%%%%%%%%%%%%%%%%%%%%%%%
\section{Numerical Experiments and Discussion}
\label{sec:numerical-contrast}

In this section, we present a numerical example to illustrate and validate our theoretical results. Specifically, we demonstrate that the velocity field and its spatial gradient exhibit localization, amplification, and oscillations in the vicinity of the first set of prescribed points, while remaining uniformly small near the second set of prescribed points. Furthermore, we investigate the choice of the viscosity coefficient and numerically observe that the desired phenomena can still be achieved with a viscosity coefficient slightly smaller than that required by our theoretical analysis. The details are presented below.
\subsection{Geometry and initial data}
We set $M=10$, $T=10$, and prescribe two groups of points:
\[
 (z_1^{(1)},z_2^{(1)})=((0,0,0),(0,0,1)),\qquad
 (z_1^{(2)},z_2^{(2)})=((1,0,1),(1,0,0)).
\]
The first group consists of the prescribed points near which both
the velocity field and its spatial gradient are expected to exhibit
localized amplification and high-frequency oscillations.
In contrast, both the velocity field and its spatial gradient
are expected to remain uniformly small near the points in the
second group throughout the prescribed time interval $[0,T]$.

We take the adiabatic exponent $\gamma=7/5$ and choose the
initial background pressure as
\[
 p_0=M^{2\gamma/(\gamma-1)}=10^7.
\]
The constant $A$ in the pressure law is chosen such that
the target angular frequency is $\omega_0=10^3$.
Consequently, we have
\[
 A=\left(\frac{\omega_0^2}{\gamma M^2}\right)^\gamma,
 \qquad
 \rho_0=\frac{\gamma p_0}{\omega_0^2}=14.
\]
The corresponding target oscillation period is
$2\pi/\omega_0\approx 0.006283$, yielding $1591$ complete
oscillation cycles within the prescribed time interval $[0,10]$.

The whole-space initial data of Theorem~\ref{thm:quasi-singularity}
are $p(0,x)=10^7$ and $v(0,x)=\nabla H_g(x)$. To specify $H_g$ for
these four points, order the points as
$(z_1^{(1)},z_2^{(1)},z_1^{(2)},z_2^{(2)})$ and use the nine
multi-indices
\[
 \mathcal J=(0,e_1,e_2,e_3,2e_1,e_1+e_2,e_1+e_3,2e_2,e_2+e_3).
\]
The jet vector $C$ has two consecutive copies of each of the blocks
\[
 \begin{aligned}
 c^{(1)}&=(-33,11,11,11,11,11,11,11,11),\\
 c^{(2)}&=(-3/2,1/2,1/2,1/2,1/2,1/2,1/2,1/2,1/2).
 \end{aligned}
\]
These are the entries in the Herglotz construction of
Corollary~\ref{thm:acoustic-profile}. The unlisted diagonal entry
$\partial_{x_3 x_3}^2 H_g$ follows from $\Delta H_g=-H_g$.

For the computation, we use $24$ Gauss--Legendre nodes in the polar
cosine and $48$ equally spaced azimuthal nodes. Let $(d_\ell,w_\ell)$,
$1\leq\ell\leq1152$, denote the resulting sphere directions and
weights. The coefficients $a_\ell$ form the minimum Euclidean norm
solution of the linear system $Ba=C$, where
\[
 B_{(j,\beta),\ell}
   =\sqrt{w_\ell}(\mathrm i d_\ell)^\beta
      \exp(\mathrm i z_j\cdot d_\ell),\qquad \beta\in\mathcal J.
\]
We compute $a$ by singular value decomposition. The maximum residual
in the prescribed jet entries is below $2\times10^{-10}$.
The velocity initial data used on the numerical box are therefore
\begin{equation}\label{eq:experiment-initial-data}
 \begin{aligned}
 p(0,x)&=10^7,\\
 v(0,x)&=\chi_b(x)\operatorname{Re}
       \sum_{\ell=1}^{1152}
       \mathrm i\,d_\ell\sqrt{w_\ell}\,a_\ell
       \exp(\mathrm i x\cdot d_\ell).
 \end{aligned}
\end{equation}
Here $b=0.25$ and the periodic box is
$\Omega_b=[-b,1+b]\times[-b,b]\times[-b,1+b]$.
The product cutoff $\chi_b$ is one on a neighborhood of the
probes and vanishes to all orders at the box boundary. Its
factors start to decrease at distance $0.145$ from the
interval $[0,1]$ in the $x_1,x_3$ directions and from $\{0\}$
in the $x_2$ direction. In the transition region a factor is
$(1+\exp((1-s)^{-1}-s^{-1}))^{-1}$, where
$s$ is the distance beyond $0.145$ divided by $b-0.145$.
We set $v_h(t,x)=v(0,x)\cos(\omega_0 t)$ and
$f(t,x)=\kappa v_h(t,x)$.

\subsection{Choice of damping and pointwise diagnostics}
We compare only $\kappa=10^6$ and $\kappa=10^7$. Both are large
relative to the target angular frequency, with
$\kappa/\omega_0=10^3$ and $10^4$, respectively. The sufficient
coefficient in Theorem~\ref{thm:quasi-singularity} takes the form
\[
 K_*(1+T)M^2\exp(K_*MT)=1100K_*\exp(100K_*)
\]
for the present parameters. The theorem chooses $K_*$ in a
uniform energy estimate. We choose the two numerical coefficients
by testing the velocity and gradient at the four probes in the
periodic computation.

With $r=\rho/\rho_0$ and $w=v-v_h$, the numerical equations are
\[
 \begin{aligned}
 \partial_t r &=-\operatorname{div}(rv),\\
 \partial_t w+\kappa w
 &=-\partial_t v_h-v\cdot\nabla v
   -\omega_0^2 r^{\gamma-2}\nabla r,\qquad
 v=v_h+w.
 \end{aligned}
\]
We use the same $60\times20\times60$ periodic grid, with
spacing $h=0.025$, in both runs. Fourth-order centered finite
differences are applied to the mass flux and the spatial
derivatives. We treat the linear relaxation in the exponential
Runge--Kutta method of Cox and Matthews~\cite{CoxMatthews2002ETD}.
The steps are $\Delta t=10/128000$ for $\kappa=10^6$ and
$\Delta t=10/40000$ for $\kappa=10^7$. The computed gradient
combines the analytic derivative of $v_h$ on the box with the
finite-difference derivative of $w$.

Write $F$ for either $v$ or $\nabla v$ and $F_a$ for its scalar
entries. The theorem compares signed values in the first group
at the target peak times $t_m=m\pi/\omega_0$ and absolute values
in the second group on the entire time interval. To make the
amplitude comparison independent of the precise peak times, we
also compute sampled peaks in the intervals
\[
 I_m=\left[t_m-\frac{\pi}{2\omega_0},
            t_m+\frac{\pi}{2\omega_0}\right]\cap[0,T].
\]
For each $m$ we select, in the first group, the time sample
that maximizes the simultaneous signed lower bound
\[
 \max_{t_n\in I_m}\min_{j,a}
       (-1)^m F_a(t_n,z_j^{(1)}).
\]
For the second group we take $\max_{t_n\in I_m,j,a}
|F_a(t_n,z_j^{(2)})|$. Figure~\ref{fig:focused-all-peaks}
places these values at their respective time samples. We also
evaluate the simultaneous signed bound at each exact target
time $t_m$ using cubic time interpolation, and report its
minimum in Table~\ref{tab:focused-damping}.

\subsection{Time traces and observed peak contrast}
Figures~\ref{fig:focused-velocity-windows} and
\ref{fig:focused-gradient-windows} show the computed
$v^{(1)}$ and $\partial_{x_1}v^{(1)}$ during three windows:
near the start, around $t=5$, and near $t=10$. For each
quantity, the first and second probe within a group are
distinguished by their line styles. The same colors identify
the two damping strengths throughout the section.

\begin{table}[tbp]
 \centering
 \small
 \caption{Velocity and gradient diagnostics over $[0,10]$ on the
 $h=0.025$ grid. The first-group entries are minima of the signed
 components at the prescribed peak times. The second-group
 entries are maxima of absolute components over all saved times.}
 \label{tab:focused-damping}
 \begin{tabular}{crrrr}
 \toprule
 & \multicolumn{2}{c}{First group}
 & \multicolumn{2}{c}{Second group}\\
 $\kappa$ & Velocity & Gradient & Velocity & Gradient\\
 \midrule
 $10^6$ & 10.9346 & 1.0588 & 0.7848 & 8.8669\\
 $10^7$ & 10.9835 & 10.7640 & 0.5098 & 0.7479\\
 \bottomrule
 \end{tabular}
\end{table}

For $\kappa=10^6$, all velocity entries in the first group attain the prescribed sign and exceed $10$ in signed value at each target peak. The largest sampled absolute velocity entry in the second group is \(0.7848\). The second-group gradient maximum is \(8.8669\), which exceeds the threshold $1$.

For $\kappa=10^7$, both the velocity and the gradient meet the first-group peak thresholds and the second-group bounds in the numerical diagnostics. The respective second-group maxima are \(0.5098\) and \(0.7479\).

The smallest common sampled velocity peaks over the half-cycles are \(10.9281\) and \(10.8984\), respectively.

In the previous resolution check for $\kappa=10^6$, the
$h=0.0125$ mesh on $[0,0.0125]$ gave second-group maxima
of approximately $0.791$ for velocity and $7.635$ for gradient.
For $\kappa=10^7$, halving the time step and increasing the box
padding to $0.30$ in two separate runs gave second-group gradient
maxima of $0.745$ and $0.568$, respectively. These tests indicate
that the threshold comparisons persist in the tested
discretizations, with sensitivity of individual gradient
magnitudes to mesh and box size.

\begin{figure}[tbp]
 \centering
 \includegraphics[width=\linewidth]{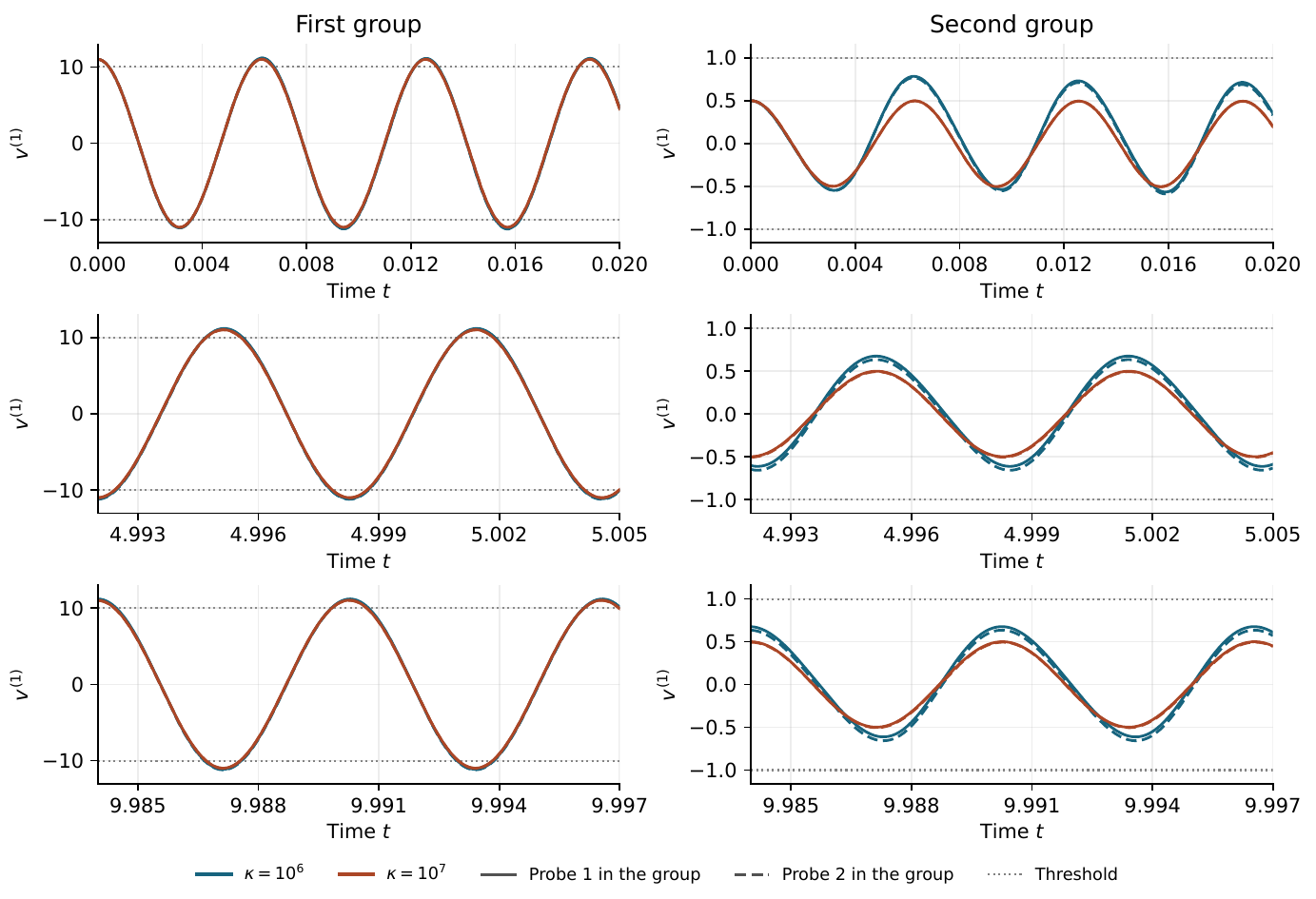}
 \caption{Computed velocity component $v^{(1)}$ at each probe in
 three time windows. Solid and dashed curves distinguish the
 first and second probes within each group, respectively. Dotted
 horizontal lines mark $\pm10$ in the first group and $\pm1$
 in the second.}
 \label{fig:focused-velocity-windows}
\end{figure}

\begin{figure}[tbp]
 \centering
 \includegraphics[width=\linewidth]{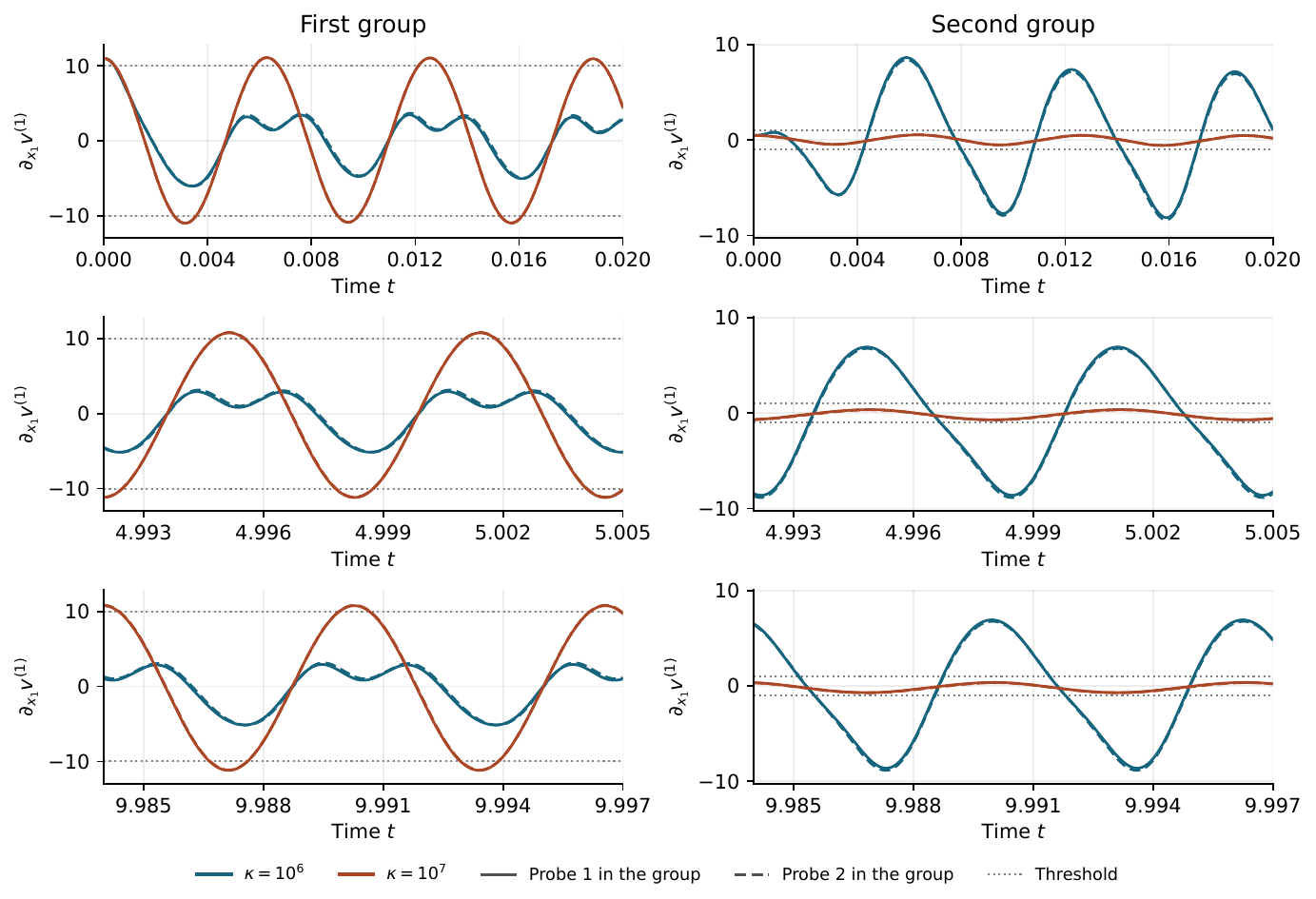}
 \caption{Computed derivative $\partial_{x_1}v^{(1)}$ at the
 four probes. The colors, line styles, and thresholds match
 Figure~\ref{fig:focused-velocity-windows}.}
 \label{fig:focused-gradient-windows}
\end{figure}

\begin{figure}[tbp]
 \centering
 \includegraphics[width=\linewidth]{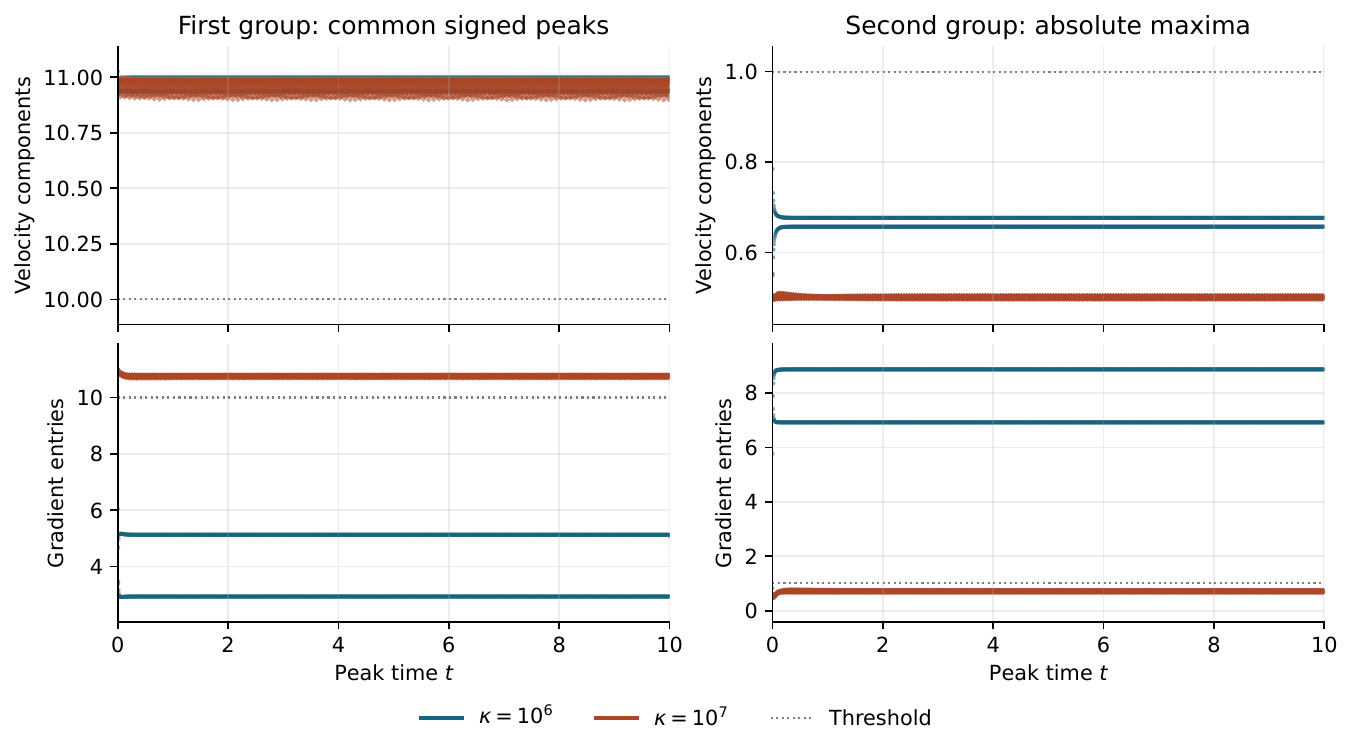}
 \caption{Computed peaks over all half-cycles in $[0,10]$.
 The top panels use all velocity components and the bottom panels
 use all gradient entries. On the left, each value is the largest
 simultaneous signed lower bound found among the samples in its
 half-cycle. On the right, each value is the largest sampled
 absolute entry within its half-cycle.}
 \label{fig:focused-all-peaks}
\end{figure}

\clearpage
\subsection{Spatial contrast near the probes}
Figures~\ref{fig:focused-spatial-early},
\ref{fig:focused-spatial-middle}, and~\ref{fig:focused-spatial-late}
show neighborhoods of the four probes in the plane $x_2=0$ at
one early, one middle, and one late target peak. The same
physical times are used for the two damping strengths. Figure~\ref{fig:focused-spatial-actual}
shows the corresponding fields at the sampled common velocity
crest in the early half-cycle for each coefficient.
For each snapshot we interpolate $w$ linearly in time between
adjacent states and evaluate $v_h$ at the specified time.
The color represents $\max_i |v^{(i)}|$ or
$\max_{i,j}|\partial_{x_j}v^{(i)}|$. Darker shades represent larger
values, and the scale for each quantity is fixed across all
figures. The numbers in the panels give the value
at the center probe. The spatial fields can vary rapidly
outside the marked points, where the probe thresholds do not
apply.

\begin{figure}[!htbp]
 \centering
 \includegraphics[width=0.94\linewidth]{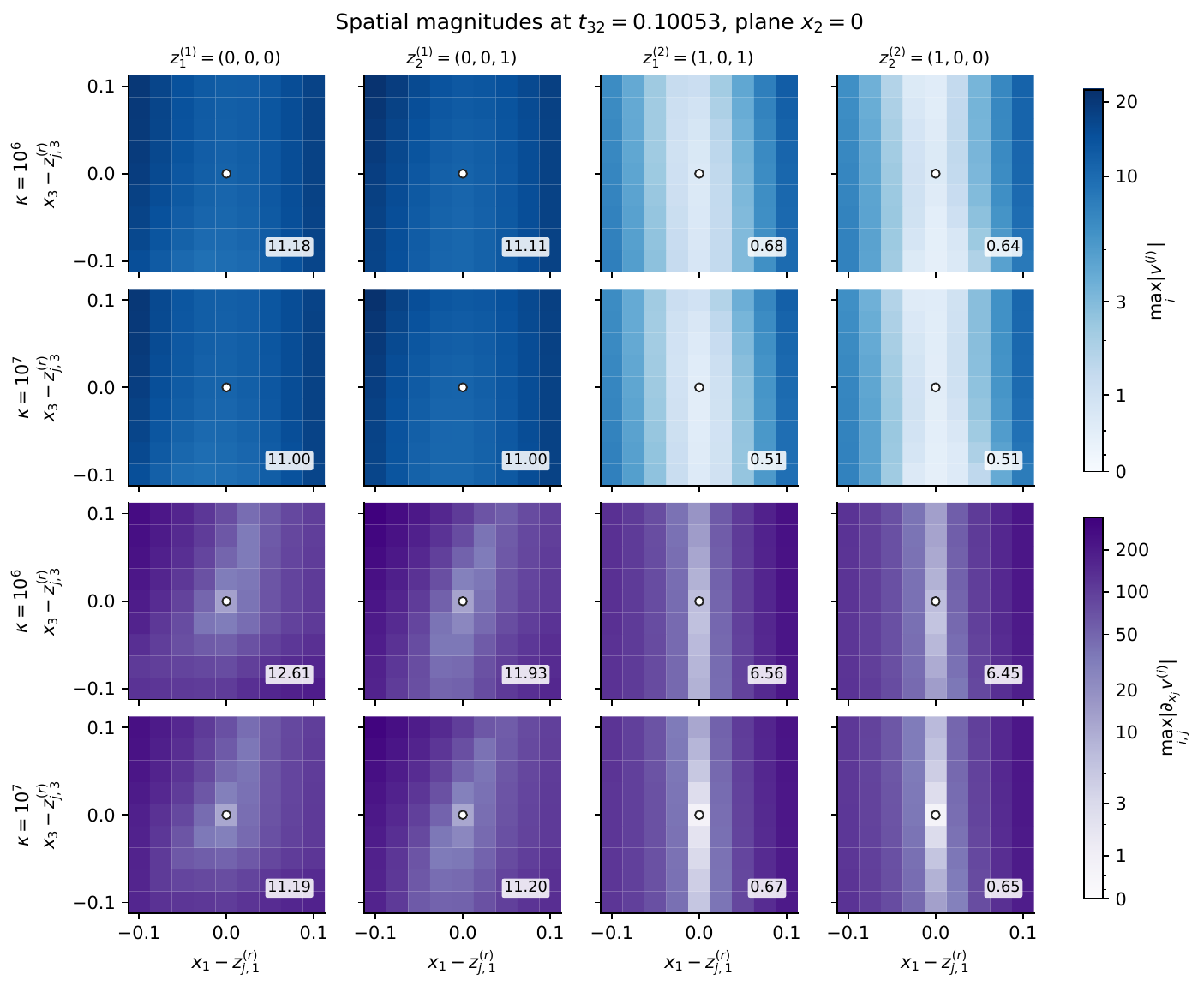}
 \caption{Spatial magnitudes in the early window at
 $t_{32}\simeq0.10053$. The upper two rows show velocity and
 the lower two show the first spatial derivatives.}
 \label{fig:focused-spatial-early}
\end{figure}

\begin{figure}[p]
 \centering
 \includegraphics[width=\linewidth]{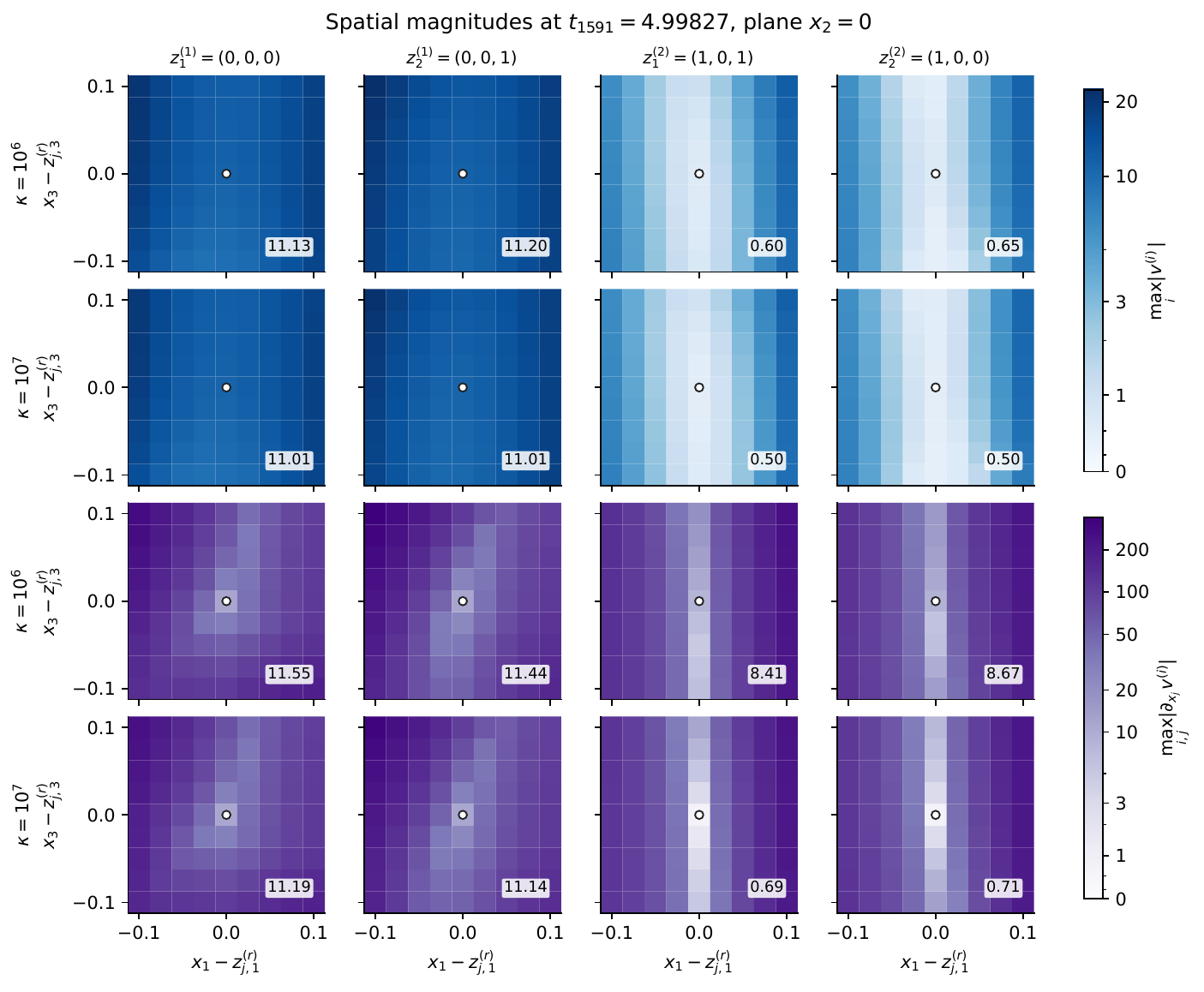}
 \caption{Spatial magnitudes near the middle of the interval at
 $t_{1591}\simeq4.99827$, a negative peak of the target.}
 \label{fig:focused-spatial-middle}
\end{figure}

\begin{figure}[p]
 \centering
 \includegraphics[width=\linewidth]{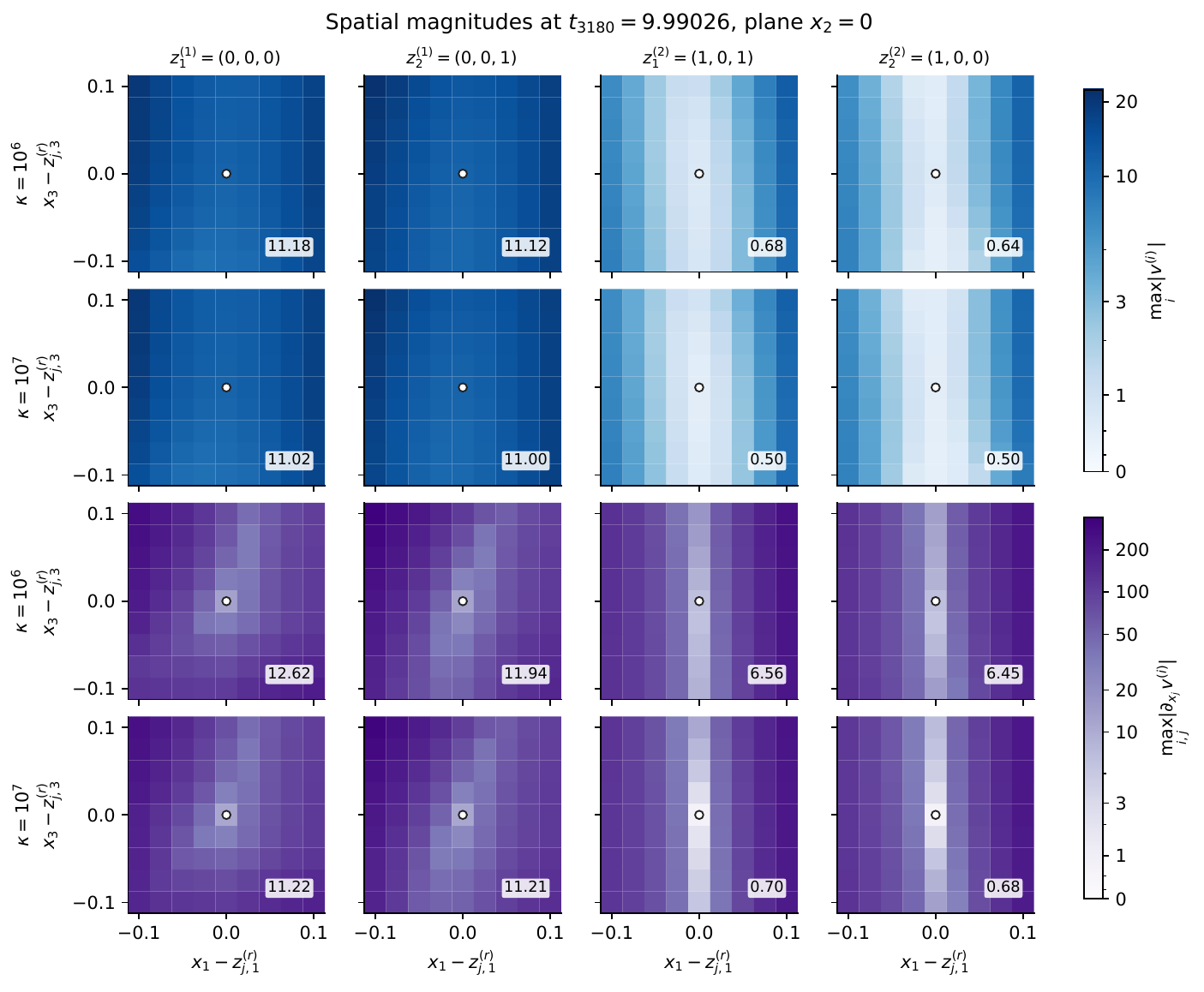}
 \caption{Spatial magnitudes near the end of the interval at
 $t_{3180}\simeq9.99026$. The color scales agree with
 Figures~\ref{fig:focused-spatial-early} and
 \ref{fig:focused-spatial-middle}.}
 \label{fig:focused-spatial-late}
\end{figure}

\begin{figure}[p]
 \centering
 \includegraphics[width=\linewidth]{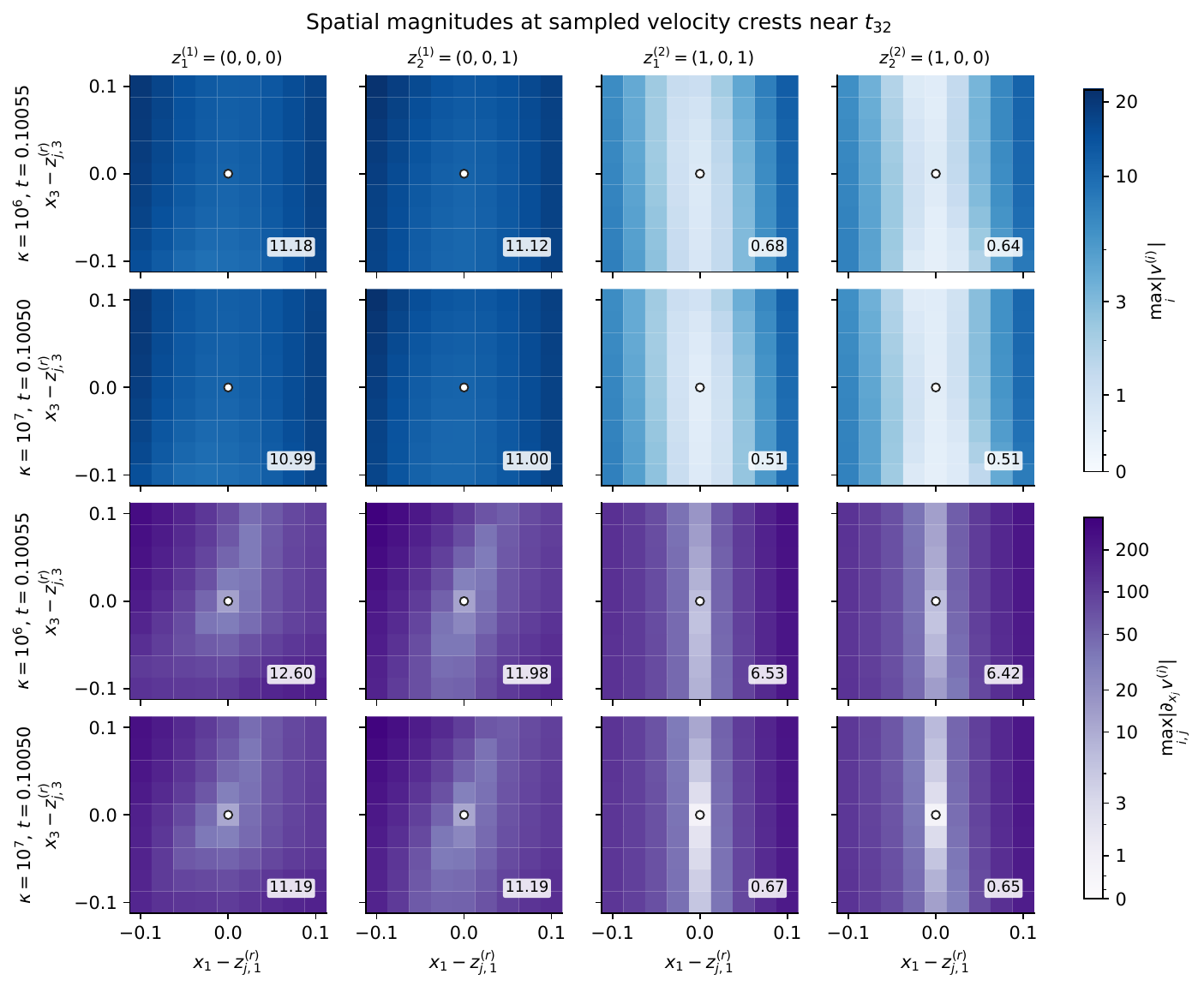}
 \caption{Spatial magnitudes at the sampled common velocity crest
 within $I_{32}$ for each damping coefficient. The crest time is
 chosen from the computed velocities in the first group. Both
 quantities use the same color scales as the prescribed-phase
 snapshots.}
 \label{fig:focused-spatial-actual}
\end{figure}

\clearpage

\subsection*{Acknowledgements}
The work of H. Liu is supported by the Hong Kong RGC General Research Funds (projects 11304224, 11311122, and 11303125), the NSFC/RGC Joint Research Fund (project N\_CityU101/21), and the France--Hong Kong ANR/RGC Joint Research Grant (project A-CityU203/19). The authors also acknowledge the use of ChatGPT and Gemini for language polishing and assisting in checking the proofs of certain auxiliary lemmas; all mathematical statements and arguments have been thoroughly verified by the authors, who take full responsibility for the content of the manuscript.
%%%%%%%%%%%%%%%%%%%%%%%%%%%%%%%%%%%%%%%%%%%%%%%%%%
%%%%%%%%%%%%%%%%%%%%%%%%%%%%%%%%%%%%%%%%%%%%%%%%%%
\section*{Declarations}

\subsection*{Conflict of Interest}
The authors declare that they have no conflict of interest to disclose.

\subsection*{Data Availability Statement}
The source package for this manuscript contains the simulation code, Herglotz coefficients, sampled peak diagnostics, and numerical figures for Section~\ref{sec:numerical-contrast}. The accompanying scripts reproduce the probe histories and spatial profiles.

%%%%%%%%%%%%%%%%%%%%%%%%%%%%%%%%%%%%%%%%%%%%%%

\end{document}